\documentclass[12pt,a4paper]{amsart}
\usepackage[utf8]{inputenc}
\usepackage[T1]{fontenc}
\usepackage{amsmath}
\usepackage{amsthm}
\usepackage{amsfonts}
\usepackage{amssymb}
\usepackage{mathrsfs}
\usepackage{url}
\usepackage{mathdots}
\usepackage{geometry}
\usepackage{verbatim}
\usepackage{hyperref}
\usepackage{float}
\usepackage{enumerate}
\usepackage{tikz-cd}
\usepackage[nolabel]{showlabels}
\usepackage{extarrows}
\usepackage{comment}
\usepackage{bbm}
\usepackage{dynkin-diagrams}

\hypersetup{
    bookmarks=true,         
    unicode=false,          
    pdftoolbar=true,        
    pdfmenubar=true,        
    pdffitwindow=false,     
    pdfstartview={FitH},    
		pdftitle={Motivic weight {$23$}},    
    pdfauthor={Ga\"etan Chenevier and Olivier Ta\"ibi},     
    colorlinks=true,       
    linkcolor=black,          
    citecolor=black,        
    filecolor=black,      
    urlcolor=black}           

\newcommand{\ol}[1]{\overline{#1}}

\newcommand{\Q}{\mathbb{Q}}

\newcommand{\R}{\mathbb{R}}
\providecommand{\C}{\mathbb{C}}
\renewcommand{\C}{\mathbb{C}}

\newcommand{\Fp}{\mathbb{F}_p}
\newcommand{\F}{\mathbb{F}}

\newcommand{\Zp}{\mathbb{Z}_p}
\newcommand{\Zell}{\mathbb{Z}_{\ell}}
\newcommand{\Z}{\mathbb{Z}}

\newcommand{\res}{\operatorname{res}}
\newcommand{\Stab}{\operatorname{Stab}}

\newcommand{\Out}{\mathrm{Out}}
\newcommand{\GL}{\mathrm{GL}}

\newcommand{\SO}{\mathrm{SO}}
\renewcommand{\O}{\mathrm{O}}
\newcommand{\rmE}{\mathrm{E}}
\newcommand{\rmQ}{\mathrm{Q}}

\newcommand{\Hom}{\mathrm{Hom}}

\newcommand{\ps}{\par \smallskip}
\newcommand{\isomo}{\overset{\sim}{\rightarrow}}

\newtheorem{theo}{Theorem}[section]

\newtheorem{lemm}[theo]{Lemma}

\newtheorem{coro}[theo]{Corollary}
\newtheorem{defi}[theo]{Definition}
\newtheorem{prop}[theo]{Proposition}

\newtheorem{theointro}{Theorem}

\newtheorem{fact}[theo]{Fact}
\newtheorem{corodefi}[theo]{Corollary-Definition}

\newenvironment{pf}
{\medskip\noindent {\it Proof.  }}
{\hfill\nobreak $\Box$ \par\bigbreak}

\begin{document}

\author{Ga\"etan Chenevier}
\address[Ga\"etan Chenevier]{CNRS, Universit\'e Paris-Sud}
\author{Olivier Ta\"ibi}
\address[Olivier Ta\"ibi]{CNRS, \'Ecole Normale Sup\'erieure de Lyon}
\thanks{Ga\"etan Chenevier and Olivier Ta\"ibi are supported by the C.N.R.S. and by the project ANR-14-CE25.}
\title[Siegel modular forms of weight $13$ and the Leech lattice]{Siegel modular forms of weight $13$ and the Leech lattice}

\begin{abstract} For $g=8,12,16$ and $24$, there is a nonzero alternating $g$-multilinear form on the ${\rm Leech}$ lattice, unique up to a scalar, which is invariant by the orthogonal group of ${\rm Leech}$. The harmonic Siegel theta series built from these alternating forms are Siegel modular cuspforms of weight $13$ for ${\rm Sp}_{2g}(\Z)$. We prove that they are nonzero eigenforms, determine one of their Fourier coefficients, and give informations about their standard ${\rm L}$-functions. These forms are interesting since, by a recent work of the authors, they are the only nonzero Siegel modular forms of weight $13$ for ${\rm Sp}_{2n}(\Z)$, for any $n\geq 1$.
\end{abstract}

\maketitle

\section*{Introduction}

	Let $L$ be an even unimodular lattice of dimension $24$.	We know since Conway and Niemeier that either $L$ has no root, and is isomorphic to the Leech lattice (denoted ${\rm Leech}$ below), or $L \otimes \R$ is generated by the roots of $L$ \cite[Chap. 16]{splag}. In the latter case it follows that for any integer $g \geq 1$ there is no nonzero alternating $g$-form on $L$ which is invariant by its orthogonal group ${\rm O}(L)$ (see \S \ref{complements}). On the other hand, ${\rm O}({\rm Leech})$ is Conway's group ${\rm Co}_0$ and a  computation made in \cite{chenevier_niemeier}, using the character $\chi_{102}$ and the power maps given in the $\mathbb{ATLAS}$ \cite{atlas}, revealed that the average characteristic polynomial of an element of ${\rm O}({\rm Leech})$ is
\begin{equation}\label{invaltleech} \frac{1}{|{\rm Co}_0|}  \sum_{\gamma \in {\rm Co}_0} \,\det(t- \gamma)\, =\, t^{24} + t^{16} + t^{12} + t^8 + 1.\end{equation}
It follows that for $g$ in $\{8,12,16,24\}$, and only for those values of $g \geq 1$, there is a nonzero alternating $g$-multilinear form, unique up to a rational scalar,
$$\omega_g \,:\, {\rm Leech}^g \,\longrightarrow \,\Q$$
such that $\omega_g(\gamma v_1, \gamma v_2,\dots,\gamma v_g)= \omega_g(v_1,v_2,\dots,v_g)$ for all $\gamma \in {\rm O}({\rm Leech})$ and all $v_1,\dots,v_g$ in ${\rm Leech}$.\ps

A first natural question is to exhibit concretely these $\omega_g$. Of course, we may choose for $\omega_{24}$ the determinant taken in a $\Z$-basis of ${\rm Leech}$: it is indeed ${\rm O}({\rm Leech})$-invariant as we know since Conway \cite{conwayoleech} that any element in ${\rm O}({\rm Leech})$ has determinant $1$, a non trivial fact. We will explain in \S \ref{constomegag} a simple and uniform construction of $\omega_8, \omega_{12}$ and $\omega_{16}$. It will appear that it is not an accident that the numbers $0,8,12,16$ and $24$ are also the possible length of an element in the extended binary Golay code. \ps

A second interesting question is to study the Siegel theta series
\begin{equation} \label{formfg} {\rm F}_g \,\,\overset{\rm def}{=}\, \,  \sum_{v \in {\rm Leech}^g} \omega_g(v) \, q^{\frac{v\cdot v}{2}}.\end{equation}
Here $v\cdot v$ abusively denotes the Gram matrix $(v_i \cdot v_j)_{1 \leq i,j \leq g}$ with $v=(v_1,\dots,v_g)$, and $q^{n}$ abusively denotes the function $\tau \mapsto \, e^{ \,2\,\pi\, i\, {\rm Tr} (n \tau)}$ for $\tau \in {\rm M}_g(\C)$ in the Siegel upper-half space. This theta series is a Siegel modular form of weight $13$ for the full Siegel modular group ${\rm Sp}_{2g}(\Z)$, necessarily a cuspform, whose Fourier coefficients are in $\Q$. The first paragraph above even shows that this is an eigenform... provided it is nonzero ! (see \S \ref{complements}). \ps

Among these four forms, only ${\rm F}_{24}$ seems to have been studied in the
past, by Freitag, in the last section of \cite{Freitag_harm_theta}.
He observed that ${\rm F}_{24}$ is indeed a nonzero eigenform.
Indeed, if we choose $\omega_{24}$ as above, and if $u \in {\rm Leech}^{24}$ is
a $\Z$-basis of ${\rm Leech}$ with $\omega_{24}(u)=1$, there are exactly $|{\rm
O}({\rm Leech})|$ vectors $v \in {\rm Leech}^{24}$ with $v\cdot v = u \cdot u$,
namely the $\gamma u$ with $\gamma$ in ${\rm O}({\rm Leech})$.
They all satisfy $\omega_{24}(v) = 1$ since any element of ${\rm O}({\rm
Leech})$ has determinant $1$.
It follows that the Fourier coefficient of ${\rm F}_{24}$ in $q^{\frac{u \cdot
u}{2}}$ is $|{\rm O}({\rm Leech})|$, it is thus nonzero. \ps

Nevertheless, the following theorem was recently proved in \cite[Cor. 1 \& Prop. 5.12]{CheTai}:

\begin{theointro}\label{thmchetai} For $g\geq 1$ the space of weight $13$ Siegel modular forms for ${\rm Sp}_{2g}(\Z)$ is $0$, or we have $g  \in \{8,12,16,24\}$, it has dimension $1$, and is generated by ${\rm F}_g$.
\end{theointro}

The proof given {\it loc. cit.} of the non vanishing of the forms ${\rm F}_g$ is quite indirect. Using quite sophisticated recent results from the theory of automorphic forms (Arthur's classification \cite{Arthur_book}, recent description by Arancibia, Moeglin and Renard of certain local Arthur packets \cite{AMR,MR_scalar}) we observed the existence of $4$ weight $13$ Siegel modular eigenforms for ${\rm Sp}_{2g}(\Z)$ of respective genus $g=8,12,16$ and $24$, and with specific standard ${\rm L}$-function. The cases $g=16$ and $g=24$ are especially delicate, and use recent results of Moeglin and Renard \cite{MR_scalar}. Using works of B\"ocherer \cite{bocherer_theta}, we then checked that they must be linear combinations of Siegel theta series construction from alternating $g$-multilinear forms on Niemeier lattices, hence must be equal to ${\rm F}_g$ by what we explained above. Our aim here is to provide a more direct and elementary proof of the non vanishing of the $3$ remaining forms ${\rm F}_g$, by exhibiting a nonzero Fourier coefficient. \ps

Let $F = \sum_n \, a_n \, q^n$ be a Siegel modular form for ${\rm Sp}_{2g}(\Z)$ of {\it odd} weight, and $N$ an even Euclidean lattice of rank $g$.
If $v$ and $v'$ in $N^g$ are $\Z$-bases of $N$, with associated Gram matrices
$2n$ and $2n'$, we have $a_n \, = \, (\det \gamma)\, a_{n'}$ where $\gamma$ is the unique element of ${\rm GL}(N)$ with $\gamma(v)=v'$.
In particular, the element $\pm a_n$  (a {\it complex number modulo sign}) only depends on the isometry class of $N$, and will be denoted $a_N(F)$ and called {\it the $N$-th Fourier coefficient of $F$}. For instance, we have $a_{{\rm Leech}}({\rm F}_{24})=\pm |{\rm O}({\rm Leech})|$. We will say that a lattice $N$ is {\it orientable} if any element of ${\rm O}(N)$ has determinant $1$; note that we have $a_N(F)=0$ for all non orientable even lattice $N$ of rank $g$. \ps

Four orientable rank $g$ even lattices ${\rm Q}_g$ with $g$ in $\{8,12,16,24\}$ will play an important role below. The lattice ${\rm Q}_{24}$ is simply ${\rm Leech}$.
The lattice ${\rm Q}_{12}$ is the unique even lattice $L$ of rank $12$ without
roots with $L^\sharp/L \simeq (\Z/3\Z)^6$; it is also known as the {\it
Coxeter-Todd lattice} \cite[Ch. 4\, \S 9]{splag}.
The lattices ${\rm Q}_{8}$ and ${\rm Q}_{16}$ are the unique even lattices $L$ without roots, of respective rank $8$ and $16$, with $L^\sharp/L \simeq (\Z/5\Z)^4$; the lattice ${\rm Q}_8$ was known to Maass and is sometimes called the {\it iscosian lattice} \cite[Ch. 8 \,\S 2]{splag}. These properties, and other relevant ones for our purposes, will be reviewed or proved in \S \ref{fixedpoints} and \S \ref{sectlattices}. An important one is that there is a unique ${\rm O}({\rm Leech})$-orbit of sublattices of ${\rm Leech}$ isometric to ${\rm Q}_g$. Our main result is the following. \ps

\begin{theointro}\label{mainthm}For each $g$, the ${\rm Q}_g$-Fourier coefficient of ${\rm F}_g$ is nonzero. More precisely, if we normalize $\omega_g$ as in Definition \ref{defomegag}, we have $$a_{{\rm Q}_g}({\rm F}_g) \, = \, \pm \, {\rm n}_g \, {\rm e}_g,$$
where ${\rm n}_g$ is the number of isometric embeddings ${\rm Q}_g \hookrightarrow {\rm Leech}$, and with ${\rm e}_8={\rm e}_{16}=5$, ${\rm e}_{12}=18$ and ${\rm e}_{24}=1$.
\end{theointro}

As we will see, the quantity ${\rm e}_g$ has the following conceptual explanation in terms of the extended binary Golay code $\mathcal{G}$ and its automorphism group ${\rm M}_{24}$. Write ${\rm res}\, {\rm Q}_g \simeq (\Z/p_g\Z)^{r_g}$; then ${\rm e}_g$ is the number of $g$-element subsets of $\mathcal{G}$ containing the fixed point set of a given element of ${\rm M}_{24}$ of shape $1^{24- p_g r_g}\, p_g^{r_g}$ (Lemmas \ref{geometrym24} \& \ref{calcwgbasis}). We will also prove ${\rm n}_g  = |{\rm O}({\rm Leech})|/\kappa_{24-g}$, with $\kappa_g=1$ for $g<12$, $\kappa_{12}=3$ and $\kappa_{16}=10$. We would like to stress that our proof of Theorem \ref{mainthm} does not rely on any computer calculation other than the simple summations \eqref{invaltleech} and \eqref{invaltm24}. \ps

Last but not least, we discuss in the last section the standard ${\rm L}$-functions of the eigenforms ${\rm F}_g$: see Theorem \ref{standardparamFg}. This last part is less elementary than the others, and relies on \cite{Arthur_book,AMR,TaiMult} (but not on \cite{MR_scalar}). \ps

We end this introduction by discussing prior works on the determination of the spaces ${\rm M}_k({\rm Sp}_{2g}(\Z))$ of Siegel modular forms of weight $k$ for ${\rm Sp}_{2g}(\Z)$, and its subspace ${\rm S}_k({\rm Sp}_{2g}(\Z))$ of cuspforms, for $k<13$. For this purpose, the subspace $\Theta_{n}^g$ of ${\rm M}_{n}({\rm Sp}_{2g}(\Z))$ generated by (classical) Siegel theta series of even unimodular lattices of rank $2n$ has drawn much attention, starting with Witt's famous conjecture $\dim \Theta_{8}^g = 2 \Leftrightarrow g\geq 4$, proved by Igusa. The study of $ \Theta_{12}^g$ has a rich history as well.  Erokhin proved $\dim \Theta_{12}^g=24$ for $g\geq 12$ in \cite{Erok}, and Borcherds-Freitag-Weissauer showed $\dim \Theta_{12}^{11}=23$ in \cite{BoFrWe}. Nebe and Venkov conjectured in \cite{NebeVenkov} that the $11$ integers $\dim \Theta_{12}^g$, for $g=0,\dots,10$, are respectively  given by
$$1,2,3,4,6,8,11,14,18,20\,\,\,\text{and}\,\,\,22,$$
and proved it for $g \neq 7,8,9$. Ikeda used his ``lifts''  \cite{Ikeda01,IkedaM} to determine the standard ${\rm L}$-functions of $20$ of the $24$ eigenforms in $\Theta_{12}^{12}$. The full Nebe-Venkov conjecture was finally proved by Chenevier-Lannes \cite{CheLan}, as well as the determination of the $4$ standard ${\rm L}$-functions not determined by Ikeda. Moreover, these authors show $\Theta_{12}^g =  {\rm M}_{12}({\rm Sp}_{2g}(\Z))$ for all $g \leq 12$, as well as  $\Theta_{8}^g =  {\rm M}_{8}({\rm Sp}_{2g}(\Z))$ for all $g \leq 8$.
Simpler proof of these results, as well as their extension to all $g$, were then given in \cite{CheTai}, in which the vanishing of ${\rm S}_k({\rm Sp}_{2g}(\Z))$ is proved for $g>k$ and $k < 13$.
Let us mention that dimensions and generators of ${\rm S}_k({\rm Sp}_{2g}(\Z))$
with $g\leq k \leq 11$,  as well as standard ${\rm L}$-functions of eigenforms,
are also given in \cite{CheLan} and \cite{CheTai}, completing previous works of
several authors, including Ikeda, Igusa, Tsuyumine, Poor-Yuen and
Duke-Imamo\={g}lu.

\bigskip

\begin{center} {\sc General notations and terminology}\end{center}
Let $X$ be a set. We denote by $|X|$ the cardinality of $X$
and by $\frak{S}_X$ its symmetric group.
Let $k$ be a commutative ring.
We denote by $k\, X$ the free $k$-module over $X$.
The elements $x$ of $X$ form a natural $k$-basis of $k\, X$
that we will often denote by $\nu_x$ to avoid confusions.
For $S \subset X$ we also set $\nu_S = \sum_{x \in S} \nu_x$. \par
If $V$ and $W$ are two $k$-modules, a {\it quadratic map} $q : V \rightarrow W$ is
a map satisfying $q( \lambda v) = \lambda^2 q(v)$ for all $\lambda$ in $k$ and $v$ in $V$,
and such that $V \times V \rightarrow W, \, \, (x,y) \mapsto q(x+y)-q(x)-q(y),$
is $k$-bilinear (the  {\it associated bilinear form}). \par
A {\it quadratic space} over $k$ is a $k$-module $V$
equipped with a quadratic map (usually $k$-valued, but not always).
Such a space has an isometry group, denoted ${\rm O}(V)$, defined as the
subgroup of $k$-linear automorphisms $g$ of $V$ with $q \circ g =q$.
If $V$ is furthermore a free $k$-module of finite rank,
 and with $k$-valued quadratic form,
 the determinant of the Gram matrix of its associated bilinear form in any $k$-basis of $V$
 will be denoted by $\det V$
 (an element of $k^\times$ modulo squares). \par
  A {\it linking quadratic space} (a ${\rm qe}$-module in the terminology of \cite[Chap. 2]{CheLan}) is a finite quadratic space over $\Z$
 whose quadratic form is $\Q/\Z$-valued (or ``linking'')
 and with nondegenerate associated bilinear form.
 If $A$ is a finite abelian group, the {\it hyperbolic} linking quadratic space over $A$
 is ${\rm H}(A)=A \oplus {\rm Hom}(A,\Q/\Z)$, with the quadratic form $(x,\varphi) \mapsto \varphi(x)$. \par
 Let $L$ be a lattice in the Euclidean space $E$, with inner product $x \cdot y$.
 The {\it dual lattice} of $L$ is the lattice $L^\sharp = \{ x \in E\, \, |\, \,  x \cdot L \subset  \Z\}$.
 Assume $L$ is {\it integral}, that is $L \subset  L^\sharp$.
 A {\it root} of $L$ is an element $\alpha \in L$ with $\alpha \cdot \alpha =2$.
 The roots of $L$ form a (possibly empty) root system ${\rm R}(L)$ of type {\rm ADE} and rank $\leq \dim E$:
 see the beginning of \S \ref{sectlattices} for much more about roots and root systems.\par
 Assume furthermore  $L$ is even (that is $ x \cdot x$ is in $2 \Z$ for all $x$ in $L$).
 Then we view $L$ as a quadratic space over $\Z$ for the quadratic form $x \mapsto \frac{x\cdot x}{2}, L \rightarrow \Z$.
 Moreover, the finite abelian group $L^\sharp/L$ equipped with its nondegenerate $\Q/\Z$-valued quadratic form $x \mapsto \frac{x\cdot x}{2} \bmod \Z$
 is a linking quadratic space denoted ${\rm res}\, L$ and called the {\it residue} of $L$ (often also called the {\it discriminant group} or the {\it glue group}).

{\small
\tableofcontents
}

\section{The forms $\omega_g$}\label{constomegag}

We fix $\Omega$ a set with $24$ elements and as well as an extended binary Golay code $\mathcal{G}$ on $\Omega$.
This is a $12$-dimensional linear subspace of $(\Z/2\Z)\, \Omega$ that is often convenient to view as a subset of $\mathcal{P}(\Omega)$, the set of all subsets of $\Omega$. For any element $C$ of $\mathcal{G}$ we have $|C|=0,8,12,16$ or $24$. We first recall how to define the Leech lattice using $\mathcal{G}$, following Conway in \cite[Ch. 10, \S 3]{splag}.\ps

 An {\it octad} is an $8$-element subset of $\Omega$ belonging to $\mathcal{G}$.
 Their most important property is that any $5$ elements of $\Omega$ belong to a unique octad; in particular there are ${24 \choose 5}/{8 \choose 5}=759$ octads. We view the $24$-dimensional space $\R\, \Omega$ as an Euclidean space with orthonormal (canonical) basis the $\nu_i$ with $i$ in $\Omega$. For $S \subset \Omega$, recall that we set $\nu_S = \sum_{i \in S} \nu_i$.  Following Conway, the Leech lattice may be defined as the subgroup of $\R\, \Omega$ generated by the $\frac{1}{\sqrt{8}} \,\, 2 \nu_O$ with $O$ an octad, and the $\frac{1}{\sqrt{8}} \,\,(\nu_\Omega - 4\nu_i)$ with $i$ in $\Omega$. \ps

 The {\it Mathieu group} associated to $\mathcal{G}$ is the subgroup of $\mathfrak{S}_\Omega \simeq \mathfrak{S}_{24}$ preserving $\mathcal{G}$, and is simply denoted by ${\rm M}_{24}$. It has $48\cdot 24!/19! = 244823040$ elements. It acts on $\R\,\Omega$ (permutation representation), which realizes it a subgroup of ${\rm O}({\rm Leech})$.  We know since Frobenius the cycle decompositions, and cardinality, of all the conjugacy classes of ${\rm M}_{24}$ acting on $\Omega$ \cite[p. 12-13]{frobenius_m24}. For the convenience of the reader they are gathered in Table \ref{tableM24} below, which gives for each cycle shape the quantity $\texttt{cent}=|{\rm M}_{24}|/\texttt{card}$, where $\texttt{card}$ is the number of elements of this shape in ${\rm M}_{24}$.\footnote{We write this number in the form $n/2$ in the five cases where there are more than one conjugacy class of the given shape. In these five cases, there are exactly two conjugacy classes, each of which containing $|{\rm M}_{24}|/n$ elements.}
\begin{table}[htp]
{\tiny \renewcommand{\arraystretch}{1.8} \medskip
\begin{center}
\begin{tabular}{c|c|c|c|c|c|c|c|c|c|c}
${\texttt{shape}}$ & $1^8\,2^8$ & $2^{12}$ & $1^6\,3^6$ & $3^8$ & $2^4\,4^4$  & $1^4\,2^2\,4^4$ & $4^6$ & $1^4\,5^4$ & $1^2\,2^2\,3^2\,6^2$ & $6^4$    \\
\hline
${\texttt{cent}}$ & $21504$ & $7680$ & $1080$ & $504$ & $384$ & $128$ & $96$ & $60$ & $24$  & $24$ \\
\hline
\hline
${\texttt{shape}}$& $1^3 7^3$ & $1^2\,2\,4\,8^2$ & $2^2\,10^2$  & $1^2\,11^2$ & $2\,4\,6\,12$ & $12^2$ & $1\,2\,7\,14$ & $1\,3\,5\,15$ & $3\,21$ & $1\,23$ \\
\hline
${\texttt{cent}}$  & $42/2$ & $16$ & $20$ & $11$ & $12$ & $12$ & $14/2$ & $15/2$ & $21/2$ & $23/2$ \\
\end{tabular}
\end{center}}
\caption{{\small The cycle shape of the nontrivial elements of ${\rm M}_{24}$.} }\label{tableM24}
\end{table}
\noindent This table allows us to compute the average characteristic polynomial of an element in ${\rm M}_{24}$, and we find:

 \begin{fact}\label{invaltm24} The polynomial $\frac{1}{|{\rm M}_{24}|}  \sum_{\gamma \in {\rm M}_{24}} \,\det(t- \gamma)$ is
 $$ t^{24}\, -\, t^{23}\, - \,t^{17}\, +\,2\, t^{16}\, -\, t^{15}\, - \,t^{13}\, + \,2\, t^{12}\, -\, t^{11}\, - t^9 \,+ \,2\, t^8 \,- \,t^7 \,-\, t \,+\, 1.$$
\end{fact}

In particular, the space of ${\rm M}_{24}$-invariant alternating $g$-multilinear forms on $\Q\, \Omega$ has dimension $2$ for $g=8,12,16$. We will now exhibit concrete generators for the ${\rm M}_{24}$-invariants in each $\Lambda^g \, \Q\, \Omega$. We start with some general preliminary remarks. \par
Let $G$ be a group acting on a finite set $X$. A subset $S \subset X$ will be called {\it $G$-orientable} if the stabilizer $G_S$ of $S$ in $G$ acts on $S$ by even permutations. An {\it orientation} of such an $S$ is the choice of a numbering of its elements up to even permutations, or more formally, an $\mathfrak{A}_n$-orbit of bijections $\{1,\dots,n\} \isomo S$. Consider the permutation representation of $G$ on $\Q\, X$ and fix an integer $g \geq 1$. The dimension of the $G$-invariant subspace in $\Lambda^g\,\, \Q\, X$ is the number of $G$-orbits of $G$-orientable subsets of $X$ with $g$ elements. Indeed, fix a $G$-orientable nonempty subset $S$ of $X$ with $|S|=g$, choose an orientation $s: \{1,\dots,g\} \isomo S$, and set
\begin{equation}\label{essigmas} \beta_{s} \,= \,s(1) \wedge s(2) \wedge \dots
\wedge s(g) \hspace{.5 cm} \text{and}\hspace{.5 cm} \sigma_{s} =  \sum_{\gamma
\text{\, in\,}  G/G_S}\, \,  \gamma\, \,  \beta_{s}.\end{equation}
Then $\sigma_{s}$ is a nonzero $G$-invariant in $\Lambda^g\,\, \Q\, X$.
Both $\pm \beta_{s}$ and $\pm \sigma_{s}$ only depend on $S$, we denote them
respectively by $\beta_S$ and $\sigma_S$.
We also set $\beta_\emptyset=\sigma_{\emptyset}=1$.
It straightforward to check the following fact:\ps

\begin{fact} \label{invaltgen} If a group $G$ acts on the finite set $X$, and if $\mathcal{S}_g$ is a set of representatives for the $G$-orbits of $G$-orientable subsets of $X$ with $g$ elements, then the  $\sigma_S$ with $S$ in $\mathcal{S}_g$ are a $\Q$-basis of the $G$-invariants in $\Lambda^g\,\, \Q\, X$.
\end{fact}

\noindent The following lemma could probably be entirely deduced from Conway's results in \cite[Chap. 10 \S 2]{splag}. We will rather use Facts \ref{invaltm24} \& \ref{invaltgen} to prove it. Recall that we identify $\mathcal{P}(\Omega)$ with $(\Z/2\Z)\, \Omega$. In particular for $S_1, S_2$ in $\mathcal{P}(\Omega)$ we have $S_1 + S_2 = (S_1 \cup S_2) \smallsetminus (S_1 \cap S_2)$.\ps

\begin{lemm}\label{m24or} Let $S$ be a subset of $\Omega$.  Then $S$ is ${\rm M}_{24}$-orientable if, and only if, it is of the form $C+P$, with $C$ in $\mathcal{G}$ and either $|P|\leq 1$, or $|P|=2$ and $|P \cap C|=1$.
\end{lemm}

\begin{proof}
  The elements of $\mathcal{G}$ have size $0,8,12,16$ or $24$.
  The $C+P$ with $C$ in $\mathcal{G}$ and $P$ a point thus have size
  $1,7,9,11,13,15,17$ or $23$, and the $C+P$ with $|P|=2$ and $|C\cap P|=1$ have
  size $8,12$ or $16$.
  If we can show that all of those subsets are ${\rm M}_{24}$-orientable, then Facts \ref{invaltm24} and \ref{invaltgen} will not only prove the lemma, but also that there is a single ${\rm M}_{24}$-orbit of subsets of each of these $16$ types. \ps
Fix $C$ in $\mathcal{G}$, denote by $G_C \subset {\rm M}_{24}$ its stabilizer and by $I_C$ the image of the natural morphism $G_C \rightarrow \mathfrak{S}_C$. If we have $C=0$ or $C=\Omega$, then $C$ is ${\rm M}_{24}$-orientable (${\rm M}_{24}$ is even a simple group). If $C$ is an octad, Conway showed that $I_C$ is the full alternating group of $C$, so that octads are ${\rm M}_{24}$-orientable. As $\Omega$ is ${\rm M}_{24}$-orientable, it follows that complements of octads are ${\rm M}_{24}$-orientable as well. If $C$ is a dodecad, Conway showed that $I_C$ is a Mathieu permutation group ${\rm M}_{12}$ over $C$, hence in the alternating group of $C$ as well (again, it is even a simple group), so that dodecads are ${\rm M}_{24}$-orientable. \ps
Fix furthermore a subset $P$ of $\Omega$, assuming first $|P|\leq 3$, and consider the subset $C + P$ in $\mathcal{P}(\Omega)$. If $\gamma$ in ${\rm M}_{24}$ preserves $C+P$, we have $$C+\gamma(C) = P + \gamma(P).$$
The left-hand side is an element in $\mathcal{G}$, hence so is $P+\gamma(P)$. But this last subset has at most $6$ elements, hence must be $0$. It follows that the stabilizer of $C+P$ is the subgroup of $G_C$ stabilizing $P$. If we assume furthermore either $|P|=1$, or $|P|=2$ and $|P \cap C|=1$, we deduce that the ${\rm M}_{24}$-orientability of $C$ implies that of $C+P$, and we are done.
\end{proof}

 The code $\mathcal{G}$ itself also embeds in ${\rm O}(\R\, \Omega)$ by letting the element $S$ of $\mathcal{G}$ act on $\nu_i$ by $-1$ if $i$ is in $S$, $1$ otherwise. As shown by Conway \cite[Chap. 10, \S 3, Thm. 26]{splag}, this is also a subgroup of ${\rm O}({\rm Leech})$, obviously normalized by ${\rm M}_{24}$. The subgroup of ${\rm O}({\rm Leech})$ generated by $\mathcal{G}$ and ${\rm M}_{24}$ is denoted by $N$ or $2^{12}{\rm M}_{24}$ by Conway. It will play a role in the proof of the following proposition. \ps

\begin{prop}\label{formesomegag} For all $g$ in $\{0,8,12,16,24\}$, the line of ${\rm O}({\rm Leech})$-invariants in $\Lambda^g {\rm Leech} \otimes \Q$ is generated by $\sigma_C$, where $C$ is any element of $\mathcal{G}$ with $|C|=g$. \end{prop}

\begin{proof}
  Fix $g\geq 0$ and set $V_g  \,=\, \Lambda^g \,\Q\,\Omega$.
  We have the trivial inclusions
  \[ V_g^{\rm O({\rm Leech})} \subset V_g^N \subset V_g^{{\rm M}_{24}}, \]
  the dimension of the left-hand side being given by \eqref{invaltleech}, and
  that of the right-hand side by Fact \ref{invaltm24}.
  We will show that $V_g^N$ is non-zero only for $g$ in $\{0,8,12,16,24\}$, and
  that in theses cases $V_g^N$ is generated by $\sigma_C$ for $C \in \mathcal{G}$
  with $|C|=g$ (recall from the proof of Lemma \ref{m24or} that
  $\mathrm{M}_{24}$ acts transitively on the set of such $C$'s).
  Let $S$ be an ${\rm M}_{24}$-orientable subset of $\Omega$ of the form $S=C+P$
  as in the statement of Lemma \ref{m24or}.
  If Conway's group $N$ fixes $\sigma_S$, then the element $\beta_S$ in
  \eqref{essigmas} has to be fixed by the action of $C$.
  By definition, this element of $\mathcal{G}$ acts on $\beta_S$ by
  multiplication by $(-1)^{|S\cap C|}$, so we must have $|S \cap C| \equiv 0
  \bmod 2$, hence $P=0$ or $|P|=1$ and $P \cap C = \emptyset$.
  In the latter case, the element $C' = \Omega \smallsetminus C$ of
  $\mathcal{G}$ contains $P$, so it maps $\beta_S$ to $-\beta_S$ and the
  basis $\sigma_S$ of $V_g^{{\rm M}_{24}}$ is not fixed by $N$.
We have proved  $\dim V_g^N \leq 1$ for $g$ in $\{0,8,12,16,24\}$, and $V_g^N=0$ otherwise.
Fix now $C$ in $\mathcal{G}$ and set $g=|C|$. For all $C'$ in $\mathcal{G}$ we have $|C\cap C'|\equiv 0 \bmod 2$.
This shows that $N$ acts trivially on $\sigma_C$:  we have proved $V_g^N \,= \,\Q \,\sigma_C$.
\end{proof}

The inner product ${\rm Leech} \times {\rm Leech} \rightarrow \Z$, $(x,y) \mapsto x\cdot y$, induces for each integer $g \geq 0$ an ${\rm O}({\rm Leech})$-equivariant isomorphism $\Lambda^g {\rm Leech} \otimes \Q \isomo {\rm Hom}(\Lambda^g {\rm Leech},\Q)$.
This isomorphism sends the element $v_1 \wedge v_2 \wedge \cdots \wedge v_g$, with $v_i$ in ${\rm Leech}$ for all $i$, to the alternating $g$-multilinear form on ${\rm Leech}$ defined by  $(x_1,\dots,x_g) \mapsto \det ( x_i \cdot v_j)_{1 \leq i,j \leq g}$.

\begin{defi}\label{defomegag} The element $\sigma_C$, where $C$ is any element of $\mathcal{G}$ with $|C|=g$, viewed as above as an alternating $g$-multilinear form on ${\rm Leech}$, will be denoted by $\omega_g$. It is well defined up to a sign, nonzero, and ${\rm O}({\rm Leech})$-invariant.
\end{defi}

Note that by definition, we have $\omega_0 = 1$, and $\pm \omega_{24}$ is the determinant taken in the canonical basis $\nu_i$ of $\Q \, \Omega$, or equivalently,  in a $\Z$-basis of ${\rm Leech}$ as the latter is unimodular. \ps

{For the sake of completeness, we end this section with the determination of the ring structure of the ${\rm O}({\rm Leech})$-invariants in the exterior algebra $\Lambda \,{\rm Leech} \otimes \Q$.
Denote by ${\rm m}_g$ the number of $g$-element subsets of $\mathcal{G}$.
We have ${\rm m}_0={\rm m}_{24}=1$, ${\rm m}_{8}={\rm m}_{16}=759$ and ${\rm m}_{12}=2^{12}-2-2\cdot 759=2576$.
Let us simply write $\sigma_g$ for the element $\pm \sigma_C$ with $C$ in $\mathcal{G}$ and $|C|=g$.\ps

\begin{prop}\label{anneauinv} We have $\sigma_8\, \wedge\, \sigma_8 \,\,= \,\,\pm \,30\, \sigma_{16}$ and  $\sigma_g\, \wedge\, \sigma_{24-g} \, =\, \pm \, {\rm m}_g\, \sigma_{24}$ for all $g$ in $\{0,8,12,16,24\}$. \end{prop}

\begin{proof}
  Fix $C \subset \Omega$ of size $g$, denote by $C'$ its complement, and fix $c$
  and $c'$ respective orientations of $C$ and $C'$.
  The stabilizers of $C$ and $C'$ in ${\rm M}_{24}$ coincide, call them $G$.
  We have $\sigma_C \wedge \sigma_{C'} \,=\, \pm \, \sum_{\gamma,\gamma'
  \,\,{\text{in}}\,\, {\rm M}_{24}/G} \gamma(\beta_{c}) \wedge \gamma'
  (\beta_{c'})$.
  An element in this sum is nonzero if, and only if, we have $\gamma(C) \cap
  \gamma'(C') = \emptyset$, or equivalently $\gamma'(C)=\gamma(C)$, {\it i.e.}
  $\gamma = \gamma'$.
  We conclude the second assertion by the ${\rm M}_{24}$-orientability of
  $\Omega$ and the equality $|{\rm M}_{24}/G|={\rm m}_g$.\ps
  We now determine $\sigma_8 \wedge \sigma_8$.
  Let $\mathcal{T}$ be the set of triples $(O_1,O_2,O_3)$ where the $O_i$ are octads
  satisfying $O_1 \coprod O_2 \coprod O_3 = \Omega$ ({\it ordered trios}).
  By \cite[Chap. 10, \S 2, Thm. 18]{splag}, ${\rm M}_{24}$ acts
  transitively on $\mathcal{T}$ and we have $|\mathcal{T}|\,=\,30\, {\rm m}_8$.
  Fix $(O_1,O_2,O_3)$ in $\mathcal{T}$, an orientation $o_i$ of each $O_i$, and
  denote by $S_i$ the stabilizer of $O_i$ in ${\rm M}_{24}$. As octads are ${\rm M}_{24}$-orientable, for any
  $\gamma_1,\gamma_2,\gamma_3$ in ${\rm M}_{24}$ the element
  $t(\gamma_1,\gamma_2,\gamma_3)=\gamma_1 \beta_{o_1} \wedge \gamma_2 \beta_{o_2} \wedge \gamma_3
  \beta_{o_3}$ only depends on the $\gamma_i$ modulo $S_i$. We have
\begin{equation} \label{s8s8s8} \sigma_{o_1} \wedge \sigma_{o_2} \wedge \sigma_{o_3} \, =\,\sum_{\gamma_i \in {\rm M}_{24}/S_i}  t(\gamma_1,\gamma_2,\gamma_3).\end{equation}
 Observe that $t(\gamma_1,\gamma_2,\gamma_3)$ is nonzero
   if and only if the three octads $\gamma_1(O_1)$, $\gamma_2(O_2)$ and $\gamma_3(O_3)$
  are disjoint, in which case we have $t(\gamma_1,\gamma_2,\gamma_3) = \pm t(1,1,1) = \pm \sigma_{24}$.
  There are thus exactly $|\mathcal{T}|$ nonzero terms $t(\gamma_1,\gamma_2,\gamma_3)$ in the sum \eqref{s8s8s8}.
  Fix such a nonzero term.
  The transitivity of ${\rm M}_{24}$ on $\mathcal{T}$ shows the existence of $\gamma$ in ${\rm M}_{24}$ with $\gamma
  \gamma_i \in S_i$ for each $i$.
  As $\Omega$ is ${\rm M}_{24}$-orientable, we have
  $$t(\gamma_1,\gamma_2,\gamma_3) \,=\, \gamma t(\gamma_1,\gamma_2,\gamma_3)\,=\,t(\gamma\gamma_1,\gamma\gamma_2,\gamma \gamma_3)\,=\,t(1,1,1).$$
  (``the sign is always $+1$'').
  We have proved $\sigma_8 \wedge \sigma_8 \wedge \sigma_8\, =\, \pm\,|\mathcal{T}|\,\sigma_{24}$.
  As $\sigma_8 \wedge \sigma_8$ must be a multiple of $\sigma_{16}$, we conclude
  by the identity $\sigma_8 \wedge \sigma_{16} \,=\, \pm {\rm m}_8 \,\sigma_{24}$.
\end{proof}

\section{Fixed point lattices of some prime order elements in ${\rm M}_{24}$}\label{fixedpoints}

We keep the notations of \S \ref{constomegag}, and fix an element $c$ in ${\rm M}_{24}$ of order $p$, with $p$ an odd prime. We are interested in the fixed points lattice
$$Q = \{ v \in {\rm Leech}\, \, |\, \,  c v = v\},$$
and in its orthogonal $Q^\perp$ in ${\rm Leech}$. Let $F \subset \Omega$ the subset of fixed points of $c$ and $\mathcal{Z} \subset \mathcal{P}(\Omega)$ the set of supports of its $p$-cycles. We have $a\,+\,p\,b\, =\, 24$ with $a=|F|$, $b=|\mathcal{Z}|$, and $b \geq 1$. Those lattices are special cases of those considered in \cite{haradalang}. \ps

We denote by ${\rm I}_n \otimes \Z/p\Z$ the linking quadratic space $(\Z/p\Z)^n$ equipped with $\frac{1}{p}\Z/\Z$-valued quadratic form $\frac{1}{p} \sum_{i=1}^n x_i^2$. If $V$ is a quadratic space, we denote by $-V$ the quadratic space with same underlying group but opposite quadratic form. \ps

\begin{lemm}\label{lemmedefq}
  The lattices $Q$ and $Q^\perp$ are even, without roots, of respective ranks
  $a+b$ and $(p-1)b$, and we have ${\rm res}\, Q \,\simeq {\rm I}_b\otimes
  \Z/p\Z$ and ${\rm res} \, Q^\perp \simeq - \,{\rm res}\, Q$.
\end{lemm}

\begin{proof} It is clear that $Q$ and $Q^\perp$ are even and without roots, as so is {\rm Leech}. We also have $p \,{\rm Leech} \,\subset Q \oplus Q^\perp$ because of the identity $1+c+c^2+\dots+c^{p-1} \in  p + (c-1)\Z[c]$. As ${\rm Leech}$ is unimodular and $p$ is odd, we deduce that both $\det Q$ and $\det Q^{\bot}$ are odd.
It is thus enough to prove both assertions about ${\rm res}\, Q$ and ${\rm res}\, Q^\perp$ after inverting $2$.
As $\Omega$ is the disjoint union of $3$ octads, note that the $24$ elements $\sqrt{2}\, \nu_i$ with $i \in \Omega$ form an orthogonal $\Z[\frac{1}{2}]$-basis of ${\rm Leech}[\frac{1}{2}]$. \par
On the one hand, this implies that the $a$ elements $\sqrt{2}\, \nu_i$ with $i \in F$, and the $b$ elements $\sqrt{2} \, \nu_Z$ with $Z \in \mathcal{Z}$, form an orthogonal $\Z[\frac{1}{2}]$-basis of $Q[\frac{1}{2}]$. For the quadratic form ${\rm q}(x)=\frac{x\cdot x}{2}$ and $S \subset \Omega$, we have ${\rm q}(\sqrt{2}\, \nu_S)=|S|$: we have proved the assertion about ${\rm res}\, Q$. \par
  On the other hand, this also shows that $Q^\perp[\frac{1}{2}]$ is the
  submodule of ${\rm Leech}[\frac{1}{2}]$ consisting of the $\sum_{i \in \Omega
  \smallsetminus F} x_i \, \,\sqrt{2} \, \nu_i$ with $x_i \in \Z[\frac{1}{2}]$
  satisfying $\sum_{i \in Z} x_i =0$ for any $Z$ in $\mathcal{Z}$.
  In other words $\frac{1}{\sqrt{2}}\, Q^\perp[\frac{1}{2}]$ is isomorphic to the
  root lattice ${{\rm A}_{p-1}}^b$ over $\Z[\frac{1}{2}]$.
  It follows that $\res Q^\perp[\frac{1}{2}]$ is isomorphic to $-({\rm I}_b
  \otimes \Z/p\Z)$.
  (See also \cite[Prop. B.2.2 (d)]{CheLan} for a more conceptual proof of ${\rm
  res} \, Q^\perp \simeq - \,{\rm res}\, Q$).
\end{proof}

By Table \ref{tableM24}, there are $8$ conjugacy classes of elements of odd prime order in ${\rm M}_{24}$, with respective shape $3^8$, $1^6 \, 3^6$, $1^4\, 5^4$, $1^3\, 7^3$ (two classes), $1^2\, 11^2$ and $1 \, 23$ (two classes). For our applications we are looking for cases $1^a\, p^b$ with $a+b$ in $\{8,12,16\}$ and $Q$ orientable. Only the first three conjugacy classes just listed meet the first condition, and the class with shape $3^8$ does not meet the second. Indeed, in this case, the description above of $Q[\frac{1}{2}]$ shows $x \cdot x \equiv 0 \bmod 3$ for all $x \in Q$. This implies that $\frac{1}{\sqrt{3}} \,Q$ is an even unimodular lattice of rank $8$, necessarily isomorphic to ${\rm E}_8$, hence non orientable. In \S \ref{sectlattices}, we will check that the lattice $Q$ is actually orientable for the two remaining classes $1^6\,3^6$ and $1^4 5^4$, and has the following properties:
\ps

\begin{prop}\label{latticechar} Let $g$ be $8$, $12$ or $16$. Up to isometry, there is a unique even lattice ${\rm Q}_g$ of rank $g$ without roots and with residue isomorphic to
${\rm I}_4 \otimes \Z/5\Z$ {\rm (}case $g=8,16${\rm )} or to ${\rm I}_6 \otimes \Z/3\Z$ {\rm (}$g=12${\rm )}. The lattice ${\rm Q}_g$ is orientable, and there is a unique ${\rm O}({\rm Leech})$-orbit of sublattices of ${\rm Leech}$ isometric to ${\rm Q}_g$.
\end{prop}

In the remaining part of this section we explain how to deduce Theorem \ref{mainthm} from Proposition \ref{latticechar} (this proposition will only be used at the end, and not in the proof of the two following lemmas). \ps

Recall that a {\it dodecad} is an element of $\mathcal{G}$ with $12$ elements. Moreover, a subset $S \subset \Omega$ with $|S|=4$ (resp. $|S|=6$) is called a {\it tetrad} (resp. an {\it hexad}). Following Conway, we will also say that an hexad is {\it special} if it is contained in an octad, and {\it umbral} otherwise.  The umbral hexads are obtained as follows: choose $5$ points in an octad and $1$ in its complement. \ps

\begin{lemm}\label{geometrym24} \begin{itemize} \item[(i)] A tetrad $T$ is contained in exactly $5$ octads. \ps
\item[(ii)] If $\gamma$ in ${\rm M}_{24}$ is an element of order $5$ whose set of fixed points is a tetrad $T$, then the $5$ octads containing $T$ are permuted transitively by $\gamma$, and each of them intersects each orbit of $\gamma$ at exactly one point.\ps
\item[(iii)] An umbral hexad $U$ is contained in exactly $18$ dodecads; these $18$ dodecads are permuted transitively by the stabilizer  of $U$ in ${\rm M}_{24}$.\ps
\item[(iv)] Let $\gamma$ in ${\rm M}_{24}$ be an element of order $3$ with $6$ fixed points. The set $U$ of fixed points of $\gamma$ is an umbral hexad, and each dodecad containing $U$ intersects each orbit of $\gamma$ at exactly one point. Moreover, the stabilizer ${\rm G}_U$ of $U$ in ${\rm M}_{24}$ coincides with the normalizer of $\langle \gamma \rangle$ in ${\rm M}_{24}$, and the natural map ${\rm G}_U \rightarrow \mathfrak{S}_U$ is surjective with kernel $\langle \gamma \rangle$. \ps
\end{itemize}
\end{lemm}

Most of these statements are certainly well-known. We will explain how to deduce them from the exposition of Conway in \cite[Chap. 10 \S 2]{splag}.\ps

\begin{proof} Proof of (i). Recall that any $5$-element subset of $\Omega$ is contained in a unique octad.
This shows that if $T$ is a tetrad, its complement is the disjoint union of $5$ other tetrads $T_i$, uniquely determined by the property that $T \cup T_i$ is an octad for each $i$
(these six tetrads, namely $T$ and the $T_i$, form a {\it sextet} in the sense of Conway).  \ps

Proof of (ii). The element $\gamma$ permutes the five $T_i$ above since we have
$\gamma(T)=T$. Assume there is some $i$, some $x$ in $T_i$, and $k$ in
$(\Z/5\Z)^\times,$ with $\gamma^k(x) \in T_i$. Then $\gamma^k(T \cup T_i)$ is
the unique octad containing $T \cup \gamma^k(x)$, hence equals $T \cup T_i$, and
so we have $\gamma^k(T_i) = T_i$.
But this implies $|T_i| \geq 5$: a contradiction.\ps

Proof of the first assertion of (iii). Conway shows {\it loc. cit.} that ${\rm M}_{24}$ acts transitively on the octads, on the dodecads, and $6+1$ transitively on an octad and its complement, hence transitively on the umbral (resp. special) hexads as well. There are thus $759 \cdot {8 \choose 6}=21252$ special hexads in $\Omega$, and ${24 \choose 6} - 21252 = 113344$ umbral hexads. There are also $2^{12}-2-2\cdot 759 = 2576$ dodecads. Fix a dodecad $D$.
For any octad $O$, we have $|D+O| \in \{0,8,12,16,24\}$ since $D+O$ is in $\mathcal{G}$, and $|D+O| = 20-2|D \cap O|$, so $|D \cap O|$ is in $\{2,4,6\}$.
Therefore the octad $O$ containing any given $5$-element subset of $D$ has the property that $O \cap D$ is a special hexad.
In other words, {\it any $5$-element subset of $D$ is contained in a unique special hexad included in $D$}.
It follows that there are ${12 \choose 5}/6 = 132$ special hexads in $D$, hence ${12 \choose 6}-132 = 792$ umbral hexads. By counting in two ways the pairs $(U,D)$ with $U$ a umbral hexad, $D$ a dodecad, and $U \subset D$, we obtain that there are $792 \cdot 2576  / 113344 = 18$ dodecads containing a given umbral hexad, as asserted. \ps

In order to prove the second assertion in (iii), we show that the pairs $(U,D)$ as above are permuted transitively by ${\rm M}_{24}$. Fix a dodecad $D$. It is enough to show that the stabilizer $H$ of $D$ in ${\rm M}_{24}$ permutes transitively the umbral hexads of $D$. But $H$ is a Mathieu group ${\rm M}_{12}$ and is sharply $5$ transitive on $D$ by Conway. In particular, $H$ permutes transitively the special hexads of $D$. Fix $S \subset D$ a special hexad
and denote by $S'$ its complement in $D$.
The stabilizer $H_S$ of $S$ in $H$ acts faithfully both on $S$ and $S'$, and $5$ transitively on $S$,
by the sharp $5$ transitivity of $H$ on $D$. The two projections of the natural morphism
$H_S \rightarrow \mathfrak{S}_S \times \mathfrak{S}_{S'}$
are thus injective, and the first one is surjective: they are both bijective.
(This is of course compatible with the equality $|{\rm M}_{12}|/132=720$.)
By numbering $S$ and $S'$, we obtain two isomorphisms $H_S \isomo \mathfrak{S}_6$.
We claim that they differ by an outer automorphism of $\mathfrak{S}_6$.
Indeed, an element of ${\rm M}_{24}$ of order two with at least $1$ fixed point
on $\Omega$ has actually $8$ fixed points by Table \ref{tableM24}, which must
form an octad (see the beginning of \S 2.2 in \cite[Ch.\ 10]{splag}).
The group $H_S$ contains an element of order $2$ with $4$ fixed points in $S$,
but its $4$ remaining fixed points cannot lie in $D$ because no octad is
contained in $D$.
This proves the claim.
It follows that the stabilizer in $H_S$ of a point $P$ of $S$ (isomorphic to
$\mathfrak{S}_5$) acts transitively on $S'$, hence on the set of umbral hexads
in $D$ containing $S \smallsetminus P$.
Together with the fact that $H$ acts $5$ transitively on $D$, this shows that
$H$ acts transitively on the umbral hexads in $D$. \ps

Proof of (iv). If $O$ is an octad containing $U$, necessarily unique, we have
$\gamma(O)=O$, and so $\gamma$ stabilizes the two-element set $O \smallsetminus
U$ without fixed point: a contradiction.
So $U$ is an umbral hexad. For any $u$ in $U$, there is a unique octad $O_u$
containing $U \smallsetminus \{u\}$.
The six $O_u$, and the six $3$-element sets $Z_u=O_u \smallsetminus U$ are thus
preserved by any element of ${\rm M}_{24}$ fixing $U$ pointwise.
In particular, the $Z_u$ are the supports of the $3$-cycles of $\gamma$. The assertion about dodecads follows as we already explained in the proof of (iii) that any octad $O$ containing five points of a dodecad $D$ satisfies $|O \cap D|=6$.
This also shows that the pointwise stabilizer of $U$ in ${\rm M}_{24}$ is
$\langle \gamma \rangle$: a non trivial element of ${\rm M}_{24}$ with at least
$7$ fixed points has shape $1^8\,2^8$ by Table \ref{tableM24}, and as recalled
above the set of its fixed points is an octad.
Let now $G_U$ be the stabilizer of $U$ in ${\rm M}_{24}$, and $H$ the normalizer
of $\langle \gamma \rangle$.
We have $H \subset G_U$.
We know that $G_U$ has $|{\rm M}_{24}|/113344=2160$ elements.
Table \ref{tableM24} also shows that the centralizer of $\gamma$ has $1080$
elements, and that its normalizer contains an element sending $\gamma$ to
$\gamma^{-1}$, so we have $H=G_U$.
We have seen that the kernel of $G_U \rightarrow \frak{S}_U$ is $\langle \gamma
\rangle$, and we conclude that this morphism is surjective by the equality
$2160/3 = 6!$.
\end{proof}

\begin{lemm} \label{calcwgbasis}
  \begin{itemize}
    \item[(i)] Assume $c$ has shape $1^4\, 5^4$, so that $Q$ and $Q^\perp$ have
      respective ranks $8$ and $16$, and fix $v \in Q^8$ and $u \in
      (Q^\perp)^{16}$ two $\Z$-bases of these respective lattices.
      Then we have $\omega_8(v) = \pm 5$ and $\omega_{16}(u) = \pm 5$. \ps
  \item[(ii)] Assume $c$ has shape $1^6\, 3^6$, so that $Q$ has rank $12$, and
    fix $v \in Q^{12}$ a $\Z$-basis of $Q$.
    Then we have $\omega_{12}(v) = \pm 18 $.
  \end{itemize}
\end{lemm}

\begin{proof} We first show $\omega_8(v)=\pm 5$ in (i) and $\omega_{12}(v)=\pm 18$ in (ii). If $v'=(v'_1,\dots,v'_g)$ is any $\Q$-basis of $Q \otimes \Q$, we have $\omega_g(v') \,=\,\det_v(v')\, \omega_g(v)$, and $|\det_v(v')|$ is the covolume of the lattice $\sum_i \Z v'_i$ divided by the covolume of $Q$ (that is, by $25$ or $27$). Fix from now on a basis $v'$ made of the $\sqrt{2}\, \nu_i$ with $i$ in $F$, and the $\sqrt{2}\, \nu_Z$ with $Z$ in $\mathcal{Z}$.
  We have $\det_v(v')\,=\,\pm \,2^{g/2}$, so we need to prove that $2^{-g/2}
  \,\omega_g(v')$ is $\pm 5$ in the case $g=8$, and $\pm 18$ in the case
  $g=12$.\ps
By Definition \ref{defomegag}, $\omega_g(v')$ is a sum of terms of the form $\det \,(v'_i \cdot x_j)_{1 \leq i,j \leq g}$ where $\{x_1,\dots,x_g\}$ runs over all the possible elements $C$ of $\mathcal{G}$ of size $g$, numbered in an ${\rm M}_{24}$-equivariant way. For such a determinant to be nonzero, each linear form $v \mapsto v \cdot x_i$ has to be nonzero on $Q$: the subset $C$ has thus to contain all the elements of $F$, and a point in each $Z$ in $\mathcal{Z}$. In other words, such a $C$ has to meet each of the $g$ orbits of $c$ in exactly one point. Denote by $\mathcal{C}(c)$ the set of elements of $\mathcal{G}$ of size $g$ with this property. For all $C=\{x_1,\dots,x_g\}$ in $\mathcal{C}(c)$ we have
\begin{equation} \label{calcdet} \det \,(v'_i \cdot x_j)_{1 \leq i,j \leq g} \,=\, \pm \,2^{g/2}.\end{equation}
\noindent By Lemma \ref{geometrym24} (ii) and (iv),  the set $\mathcal{C}(c)$ consists of $5$ octads (resp. $18$ dodecads) if $c$ has shape $1^4\, 5^4$ (resp. $1^6\, 3^6$), and the normalizer $G$ of $\langle c \rangle$ in ${\rm M}_{24}$ permutes $\mathcal{C}(c)$ transitively. If we fix $C=\{x_1,\dots,x_g\}$ in  $\mathcal{C}(c)$, we may thus find a $|\mathcal{C}(c)|$-element subset $\Gamma \subset G$ with
$$\omega_g(v') = \pm  \sum_{\gamma \in \Gamma}  \det \,\,( v'_i \cdot \,\gamma\, x_j)_{1 \leq i,j \leq g}.$$
We claim that the $|\Gamma|$ determinants above are equal. This will show $\omega_g(v') = \pm  |\mathcal{S}(c)| 2^{g/2}$ by \eqref{calcdet}.
For any $\gamma \in G$ we have
$$\det \,(v'_i \cdot \gamma x_j)_{1 \leq i,j \leq g} \,= \, \det \,(\gamma^{-1} v'_i \cdot  x_j)_{1 \leq i,j \leq g} \, =\,\det \gamma^{-1}_{|Q}\, \, \det \,(v'_i \cdot  x_j)_{1 \leq i,j \leq g}.$$
As $Q$ is orientable by Lemma \ref{lemmedefq} and Proposition \ref{latticechar}, we have $\det \gamma_{|Q}=1$, and we are done. We may actually avoid the use of these lemma and  proposition as follows. If $c$ has shape $1^4\, 5^4$, we may choose $\Gamma = \langle c \rangle$ by Lemma \ref{geometrym24} (ii), and we clearly have $\gamma_{|Q}={\rm id}$. If $c$ has shape $1^6\, 3^6$, the proof of Lemma \ref{geometrym24} (iv) defines a natural $G$-equivariant bijection $u \mapsto Z_u$ between $U$ and $\mathcal{Z}$. For any $\gamma \in G$ we have thus $\det \gamma_{|Q}= \epsilon^2=1$, where $\epsilon$ is the signature of the image of $\gamma$ in $\mathfrak{S}_U$.  \ps

We now prove $\omega_{16}(u)=\pm 5$ in (i).
Observe first that for any oriented octad $(O,o)$, there is a sign $\epsilon$
such that for all $u_1', \dots, u_{16}'$ in $\Q \,\Omega$ we have
\begin{equation} \label{formulaw16}
  \omega_{16}(u_1', \dots, u_{16}') \,=\, \epsilon \,\omega_{24} ( \sigma_{o}
  \wedge u_1' \wedge u_2' \wedge \dots \wedge u_{16}').
\end{equation}
Indeed, the alternating $16$-form on the right is ${\rm O}({\rm
Leech})$-invariant, as both $\sigma_o$ and $\omega_{24}$ are, so it is
proportional to $\omega_{16}$.
But if $\{u_1', \dots, u_{16}'\}$ is a $16$-element subset of $\mathcal{G}$,
both sides are equal to $\pm 1$, and we are done. \par
Choose a basis $u'$ of $Q^\perp \otimes \Q$ made of $16$ elements of the form
$\sqrt{2}\, (\nu_i-\nu_{c(i)})$ with $i$ in $\Omega \smallsetminus F =
\bigsqcup_{Z \in \mathcal{Z}} Z$ (i.e.\ choose $4$ elements $i$ in each $Z \in
\mathcal{Z}$).
Comparing covolumes as in the first case of the proof, we have to show
$\omega_{16}(u')=\pm 5 \cdot 2^8$.
Apply Formula \eqref{formulaw16} to $u' = (u_1', \dots, u_{16}')$.
If $\gamma(O)$ is an octad such that $\gamma(\beta_O) \wedge u_1' \wedge u_2'
\wedge \dots \wedge u_{16}'$ is nonzero, that octad meets at most once each $Z$
in $\mathcal{Z}$.
We have $|\mathcal{Z}|=|F|=4$ and $|O|=8$, so  $\gamma(O)$ must meet each $Z$ of
$\mathcal{Z}$ in one point and contain $F$.
By Lemma \ref{geometrym24} (i) and (ii), there are $5$ such octads, permuted
transitively by $c$.
We may choose $O$ to be one of them.
We then have
\[ \omega_{16}(u'_1,\dots,u'_{16}) \,=\, \epsilon \sum_{k \,\,\text{in}\,\,
  \Z/5\Z}\,\omega_{24} (c^k \beta_{o} \wedge u'_1 \wedge u'_2 \wedge \dots
\wedge u'_{16}). \]
Now $c$ preserves $Q^\bot$ and has determinant $1$ on it (being of order $5$),
so we have $u_1' \wedge \dots \wedge u_{16}' = c^k(u_1' \wedge \dots \wedge u_{16}')$
and the sum above is $5$ times $\omega_{24} ( \beta_{o} \wedge u'_1 \wedge
u'_2 \wedge \dots \wedge u'_{16})$ by $c$-invariance of $\omega_{24}$.
An easy computation shows that we have $\omega_{24} ( \beta_{o} \wedge u'_1
\wedge u'_2 \wedge \dots \wedge u'_{16}) = \pm\, 2^8$.
\end{proof}

We are now able to prove Theorem \ref{mainthm}, assuming Proposition \ref{latticechar}.

\begin{proof} {\it (Proposition \ref{latticechar} implies Theorem \ref{mainthm})}
  Let $\mathcal{L}_g$ be the set of sublattices of ${\rm Leech}$ isometric to
  ${\rm Q}_g$.
  This set is nonempty by Lemma \ref{lemmedefq} and we fix one of its elements,
  that we denote $Q_g$.
  By Proposition \ref{latticechar}, ${\rm O}({\rm Leech})$ acts transitively on
  $\mathcal{L}_g$, so we may find an ${\rm n}_g$-element subset $\Gamma \subset
  {\rm O}({\rm Leech})$ with $\mathcal{L}_g = \Gamma \cdot Q_g$.\par
  Fix a $\Z$-basis $u_1,\dots,u_g$ of $Q_g$, and denote by $2n$ its Gram matrix.
  The $n$-th Fourier coefficient of ${\rm F}_g$ is the sum, over all the
  $g$-uples $(v_1,\dots,v_g)$ of elements of ${\rm Leech}$ with $2n=(v_i \cdot
  v_j)_{1\leq i,j \leq g}$, of $\omega_g(v_1,\dots,v_g)$.
  There are exactly ${\rm n}_g |{\rm O}(Q_g)|$ such $g$-tuples, namely the
  $(\gamma \gamma' u_1,\dots,\gamma \gamma' u_g)$ with $\gamma \in \Gamma$ and
  $\gamma' \in {\rm O}(Q_g)$.
  The ${\rm O}({\rm Leech})$-invariance of $\omega_g$, the trivial equality
  $\omega_g(\gamma' u_1,\dots, \gamma' u_g)= (\det \gamma')
  \,\omega_g(u_1,\dots,u_g)$ for $\gamma'$ in ${\rm O}(Q_g)$, and the
  property $\det \gamma'=1$ (as $Q_g$ is orientable), imply that the
  $n$-th Fourier coefficient of ${\rm F}_g$ is ${\rm n}_g |{\rm O}(Q_g)|
  \omega_g(u_1,\dots,u_g)$.
  We conclude by Lemma \ref{calcwgbasis}.
\end{proof}

\noindent

\section{Properties of the lattices ${\rm Q}_g$}\label{sectlattices}

The aim of this section is to prove Proposition \ref{latticechar}. We make first some preliminary remarks about root lattices and their sublattices. \ps

Let $R$ be a root system in the Euclidean space $V$.
We will follow Bourbaki's definitions and notations in \cite[Chap.
VI]{bourbaki_rootsystems} and assume furthermore that we have $\alpha \cdot
\alpha =2$ for all $\alpha$ in $R$.
In particular, each irreducible component of $R$ is of type ${\bf A}_l$ ($l \geq
1$), ${\bf D}_l$ ($l\geq 3$) or ${\bf E}_l$ ($l=6,7,8$), and $R$ is identified
to its dual root system, with $\alpha^\vee = \alpha$ for all roots $\alpha$.
We denote by $Q(R)$ the even lattice of $V$ generated by $R$ and by ${\rm P}(R)$ the dual lattice ${\rm Q}(R)^\sharp$, so that we have
\[ {\rm res}\, {\rm Q}(R) \, = \, {\rm P}(R)/{\rm Q}(R). \]
It is well known that the trivial inclusion $R \subset {\rm R}({\rm Q}(R))$ is
an equality.
We will simply denote by ${\rm A}_l$, ${\rm D}_l$ and ${\rm E}_l$ for ${\rm Q}(R)$ when $R$ is ${\bf A}_l$, ${\bf D}_l$ or ${\bf E}_l$ respectively.
The {\it Weyl group} of $R$ will be denoted by ${\rm W}(R)$, and the orthogonal group
of ${\rm Q}(R)$ by ${\rm A}(R)$.
The group ${\rm W}(R)$ is the subgroup of ${\rm A}(R)$ generated by the
orthogonal symmetries ${\rm s}_\alpha(x)=x - (\alpha \cdot x) \alpha$ with
$\alpha \in R$, hence acts trivially on ${\rm res}\, {\rm Q}(R)$.
It permutes simply transitively the positive root systems $R_+$ of $R$. Fix such an $R_+$, and denote by $\{\alpha_i \, |\, i \in I\}$ its simple roots.
The $\alpha_i$ form a $\Z$-basis of ${\rm Q}(R)$, whose dual basis $\varpi_i$ (the {\it fundamental weights}) is thus a $\Z$-basis of ${\rm P}(R)$.
The {\it Weyl vector} $\rho$ associated to $R_+$ is the half-sum of elements of $R_+$, it satisfies $\rho = \sum_{i \in I} \varpi_i$.\ps\ps

Assume now $R$ is irreducible of rank $\dim V = |I| = l$; we will always identify the set $I$ with $\{1,\dots,l\}$ as in Bourbaki.
The {\it highest positive root} is the unique element $\widetilde{\alpha}$ in $R_+$ satisfying $\alpha \cdot \varpi_i \leq \widetilde{\alpha} \cdot \varpi_i$ for all $i$ in $I$ and $\alpha$ in $R$.
There are unique integers $n_i>0$ for $i=1,\dots,l$ with $\widetilde{\alpha} = \sum_{i=1}^l  n_i \alpha_i$.
Let ${\rm h}(R)$ be the Coxeter number of $R$  \cite[Chap. V \& VI]{bourbaki_rootsystems}, for $h={\rm h}(R)$ we have
\begin{equation} \label{hformulae} |R|\,=\,l\,h, \hspace{.7cm} n_1\,+\,n_2\,+\,\dots\,+n_l\, =\, h-1\hspace{.7cm}\text{and}\hspace{.7cm} \rho \cdot \rho\, =\, \frac{l}{12}\,h\,(h+1).\end{equation}
Indeed, the first equality is \cite[Chap. V\, \S 6\, Thm. 1]{bourbaki_rootsystems} and the second is  \cite[Chap. VI \,\S 1\, Prop. 31]{bourbaki_rootsystems} (see also \cite[Theorem 8.4]{kostanttds}).
The last equality may either be checked case by case, using the ${\rm ADE}$ classification, or deduced from \cite{kostanttds}.
\footnote{We may argue as follows. Recall that the {\it height} of the positive root $\alpha \in R_{+}$ is ${\rm ht}(\alpha)=\rho \cdot \alpha$.
We thus have $2\, \rho \cdot \rho\, = \,\sum_{\alpha \in R_{+}} {\rm ht}(\alpha)$.
By Bourbaki's theory of the canonical bilinear form  \cite[Ch.\ VI \S
1.12]{bourbaki_rootsystems} we also have ${\rm h}(R) \rho \cdot \rho  \,=\, \sum_{\alpha \in R_+} {\rm ht}(\alpha)^2$
as $R$ is of type ${\rm ADE}$. Let ${\rm Exp}(R)$ be the set of exponents of $R$  \cite[Ch.\ V \S 6\, D\'ef. 2]{bourbaki_rootsystems}.
By Kostant {\it loc. cit.}, we have for any map $f : \Z_{\geq 1} \rightarrow \R$ the identity
$\sum_{\alpha \in R_+} f({\rm ht}(\alpha)) = \sum_{m \in {\rm Exp}(R)} F(m)$
with $F(m)=\sum_{u=1}^m f(u)$ (see \cite[p. 82]{CheLan}). We apply this to $f(x)=x$ and $f(x)=x^2$. And using the involution $m \mapsto h-m$
of ${\rm Exp}(R)$ \cite[Ch.\ V \S 6.2]{bourbaki_rootsystems}, we obtain two linear relations between $\rho \cdot \rho$ and $\sum_{m \in {\rm Exp}(R)} m^2$.
Inverting the system gives the result.}
Recall ${\rm h}({\bf A}_l)\,=\,l+1$, ${\rm h}({\bf D}_l)\,=\,2l-2$, ${\rm h}({\bf E}_6)\,=\,12$, ${\rm h}({\bf E}_7)\,=\,18$ and ${\rm h}({\bf E}_8)=30$.
Following Borel-de Siebenthal and Dynkin, the sublattice $${\rm BS}_i(R)=\{x \in {\rm Q}(R) \, | \, x \cdot \varpi_i \equiv 0 \bmod n_i\}$$
is the root lattice ${\rm Q}(R_i)$ where $R_i$ is the root system of $V$ having as a set of simple roots $-\widetilde{\alpha}$ and the $\alpha_j$ with $j \neq i$ \cite[Chap. VI, \S 4, Exercise 4]{bourbaki_rootsystems}.
The Dynkin diagram of $R_i$ is thus obtained by removing $\alpha_i$ from the
extended Dynkin diagram of $R$.
We clearly have ${\rm W}(R_i) \subset {\rm W}(R)$.
The fundamental weights of $R_i$ with respect to the simple roots above are $-\frac{1}{n_i}\varpi_i$ and the $\varpi_j - \frac{n_j}{n_i} \varpi_i$ for $j \neq i$;
in particular, the corresponding Weyl vector of $R_i$ is $\rho - \frac{h}{n_i} \varpi_i$.

\ps

Observe that for any integer $p\geq 1$, we have an ${\rm A}(R)$-equivariant
isomorphism
\begin{align*}
  {\rm P}(R) \otimes \Z/p\Z & \overset{\sim}{\longrightarrow} {\rm Hom}({\rm
  Q}(R),\Z/p\Z) \\
  \xi & \longmapsto (x \mapsto \xi \cdot x\, \bmod p).
\end{align*}
Assertion (ii) and (iii) below are Propositions 3.4.1.2 and 3.2.4.8 in
\cite{CheLan} (see also \cite{kostanttds}).

\begin{lemm} \label{lemm:orb_root_latt}
Let $R$ be an irreducible root system, $h={\rm h}(R)$, and $p \geq 1$ an integer. \ps
\begin{itemize}
\item[(i)] Each ${\rm W}(R)$-orbit in ${\rm P}(R) / p\, {\rm Q}(R)$ admits a unique representative of the form $\sum_i  m_i \varpi_i$ with $m_i \geq 0$ for all $i$ and $\sum_i m_i n_i \leq p$. \ps
\item[(ii)] The kernel of any linear form ${\rm Q}(R) \rightarrow \Z/p\Z$ with $p <  h$ contains some element of $R$.
\item[(iii)] There is a unique $\mathrm{W}(R)$-orbit
  of linear forms $\rmQ(R) \to \Z/h\Z$ whose kernel does not contain any root, namely the orbit of the form $x \mapsto \rho \cdot x \,\bmod h$.
\end{itemize}
\end{lemm}

\begin{proof}
  The set $\Pi$ of $v \in V$ with $v \cdot \alpha_i \geq 0$ for all $i$, and
  with $v \cdot \widetilde{\alpha} \leq 1$, is a fundamental domain for the
  affine Weyl group ${\rm W}_{\rm aff}(R)={\rm Q}(R) \rtimes {\rm W}(R)$ acting
  on $V$ \cite[Chap. VI \S 2]{bourbaki_rootsystems}.
  For any $\xi$ in ${\rm P}(R)$, the ${\rm W}_{\rm aff}(R)$-orbit of
  $\frac{1}{p}\xi$ meets thus $\Pi$ in a unique element: this proves (i).
  Any linear form $\varphi : {\rm Q}(R) \rightarrow \Z/p\Z$ may be written
  $\varphi(x) = \xi \cdot x \bmod p$ for some $\xi \in {\rm P}(R) /p\, {\rm
  Q}(R)$.
  Replacing $\varphi$ by $w(\varphi)$ for some $w \in {\rm W}(R)$ we may assume
  $\xi$ has the form $\sum_i  m_i \varpi_i$ with the $m_i$ as in (i).
  If the kernel of $\varphi$ contain neither the $\alpha_i$ nor
  $\widetilde{\alpha}$, we must have $m_i>0$ for all $i$ and $\sum_i m_i n_i
  <p$, and thus $h -1 = \sum_i n_i  \leq \sum_i m_i n_i < p$.
  This proves (ii). In the case $p=h$ this inequality implies $m_i=1$ for each $i$, hence $\xi = \sum_i \varpi_i = \rho$.
  For any positive root $\alpha$ in $R$ we have $0<\alpha \cdot \rho \leq \widetilde{\alpha} \cdot \rho = h-1$.
  As we have $R= {\rm R}({\rm Q}(R))$, this shows (iii).
\end{proof}

A root system $R$ is called {\it equi-Coxeter} if its irreducible components all
have the same Coxeter number, called the Coxeter number of $R$, and denoted by
${\rm h}(R)$.

 \begin{coro}\label{corequicoxh} Let $R$ be an equi-Coxeter root system of rank $l$ and Coxeter number $h$.
 Then assertion (iii) of Lemma \ref{lemm:orb_root_latt} holds and there is a unique ${\rm W}(R)$-orbit of sublattices $L \subset {\rm Q}(R)$ with no root and ${\rm Q}(R)/L \simeq \Z/h\Z$.
 These lattices are of the form $\{ x \in {\rm Q}(R)\, \, |\, \, x \cdot \rho \equiv 0 \bmod h\}$
for a Weyl vector $\rho$ for $R$.
Assuming furthermore $\rho \in {\rm Q}({\rm R})$, $h$ odd and $l(h+1) \equiv 0 \bmod 12$, they satisfy  ${\rm res}\, L\, \simeq \,{\rm H}(\Z/h\Z)\, \bot\, {\rm res}\, {\rm Q}(R)$.
\end{coro}

\begin{pf} The first assertion is a trivial consequence of (iii) of Lemma \ref{lemm:orb_root_latt} and of $\rho \cdot \alpha =1$ for a simple root $\alpha$ of $R$.
The identity $\rho \cdot \rho\,=\,l\, h(h+1)/12$ (a consequence of \eqref{hformulae}) shows that $\rho$ is a nonzero isotropic vector in ${\rm Q}(R) \otimes \Z/h\Z$, so the last assertion follows from the general Lemma \ref{lemm:sub_sup_latt_isot} below. \end{pf}

We have ${\rm res} \,{\rm A}_n \,\simeq \,\Z/(n+1)\Z$ with ${\rm q}(\overline{1}) \equiv \frac{n}{2(n+1)} \bmod \Z$,
${\rm res} \,{\rm E}_6 \,\simeq \,- {\rm res} \,{\rm A}_2$,  ${\rm res}\, {\rm E}_8\,=\,0$. As $-1$ is a square modulo $5$, Corollary \eqref{corequicoxh} implies:

\begin{corodefi} Let $R$ be either $2\,{\bf A}_4$ or $3\, {\bf A}_2$, and set $p\,=\,{\rm h}(R)$ (either $5$ or $3$) and $g \,=\, {\rm rank}\, R$ (either $8$ or $6$).
Define ${\rm Q}_g$ as the sublattice of ${\rm Q}(R)$ whose elements $x$ satisfy $x \cdot \rho \equiv 0 \bmod p$, for a fixed Weyl vector $\rho$ in ${\rm Q}(R)$. Then ${\rm Q}_g$ is an even lattice, without roots, satisfying ${\rm res}\, {\rm Q}_g\, \simeq \,{\rm res}\, {\rm E}_g \,\oplus \,{\rm H}(\Z/p\Z)^2$.
\end{corodefi}

\noindent

\begin{prop} \label{prop:uniq_sub_E} Assume either $p=5$ and $E$ is the root lattice ${\rm E}_8$, or $p=3$ and $E$ is the root lattice ${\rm E}_6$.
Up to isometry, there is a unique triple of even lattices $(A,B,C)$ with $A \subset B \subset C$, both inclusions of index $p$, $C \simeq E$ and ${\rm R}(A) = \emptyset$.
\end{prop}

\begin{pf} Set $R={\rm R}(E)$, so that we have $E={\rm Q}(R)$.
We have to show that there is a unique ${\rm W}(R)$-orbit of index $p$ subgroups $B \subset E$ such that $B$ possesses an index $p$ subgroup without roots, and that for such a $B$ there is a unique ${\rm O}(B) \cap {\rm O}(E)$-orbit of index $p$ subgroups of $B$ without roots. We claim (provocatively) that both properties follow at once from Lemma \ref{lemm:orb_root_latt} and an inspection of the extended Dynkin diagrams of ${\bf E}_8$ and ${\bf E}_6$ drawn below:

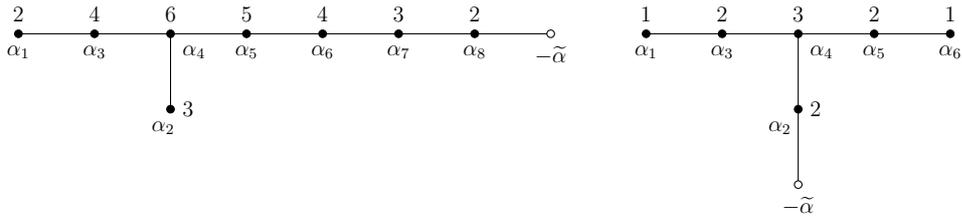
\begin{center}
$\begin{array}{ccc}
\dynkin[extended,upside down,labels*={-\widetilde{\alpha},\alpha_1,\!\!\!\!\alpha_2,\alpha_3,\alpha_4,\alpha_5,\alpha_6,\alpha_7,\alpha_8},labels={,2,3,4,6,5,4,3,2},edge length=1cm]{E}{8}
& &
\dynkin[extended,upside down,labels*={-\widetilde{\alpha},\alpha_1,\!\!\!\!\!\!\!\!\!\!\!\!\!\alpha_2,\alpha_3,\alpha_4,\alpha_5,\alpha_6},labels={,1,2,2,3,2,1},edge length=1cm]{E}{6}
\end{array}$
\end{center}
(Each simple root $\alpha_i$ is labelled with the integer $n_i$.)
Indeed, assume for instance $R \simeq {\bf E}_8$ and $p=5=n_5$.
Note that the irreducible root systems with Coxeter number $\leq 5$ are the ${\bf A}_l$ with $1 \leq l \leq 4$, so by assertion (ii) of the lemma, the irreducible components of ${\rm R}(B)$ must have this form.
On the other hand, assertion (i) asserts that for a suitable choice of a positive system of $R$ the lattice $B$ is the kernel of $x \mapsto \xi \cdot x \bmod 5$ with $\xi=\sum_i m_i \varpi_i$ and $\sum_i m_i n_i  \leq 5$.  Consider the set
$$J = \{ \,\, j\, \, |\, \,  m_j \neq 0\}.$$
We must have $|J| \leq 2$ (note $n_i\geq 2$ for all $i$) and $\alpha_j \in {\rm R}(B)$ for $j \notin J$.
An inspection of the Dynkin diagram of ${\bf E}_8$ shows that in the case $|J|=2$, we have $J \subset \{1,2,7,8\}$ and $\{2,7\} \not \subset J$,
and ${\rm R}(B)$ contains an irreducible root system of rank $5$: a contradiction.
So we have $|J|=1$ and $J \neq \{4\}$.
But this clearly implies $J=\{5\}$ and $\xi=\omega_5$ by another inspection of this diagram. So $B$ is the Borel-de Siebenthal lattice ${\rm BS}_5(R)={\rm Q}(R_5)$,
and is isomorphic to the root lattice ${\rm A}_4 \oplus {\rm A}_4$.
Note that we have ${\rm h}({\rm A}_4)=5=p$.
By the last assertion of Lemma \ref{lemm:orb_root_latt} applied to $R_5$, there
is a unique ${\rm W}(R_5)$-orbit of index $5$ sublattices of ${\rm Q}(R_5)$
without root.
As we have ${\rm W}(R_5) \subset {\rm W}(R)$, this concludes the proof in the case $R \simeq {\bf E}_8$.
The case $R \simeq {\bf E}_6$ is entirely similar.
\end{pf}

\begin{prop}\label{coreg} Let $(g,p,m)$ be either $(8,5,4)$ or $(6,3,5)$.
Up to isometry, ${\rm Q}_g$ is the unique even lattice of rank $g$ without roots satisfying ${\rm Q}_g^\sharp /{\rm Q}_g \simeq (\Z/p\Z)^m$. \par
Moreover, ${\rm O}({\rm Q}_g)$ permutes transitively the totally isotropic planes $($resp. lines, resp. flags$)$ of ${\rm res}\, {\rm Q}_g$. The inverse image in ${\rm Q}_g^\sharp$ of such an isotropic plane $($resp. line$)$ is isometric to ${\rm E}_g$ $($resp. to ${\rm A}_4 \oplus {\rm A}_4$ for $g=8$, to ${\rm A}_2 \oplus {\rm A}_2 \oplus {\rm A}_2$ for $g=6)$.
\end{prop}

In the statement above, by a {\it totally isotropic flag} of ${\rm res}\, {\rm Q}_g$ we mean a pair $(D,P)$ with $D$ a line and $P$ a totally isotropic plane containing $D$.

\begin{pf}  Let $A$ be an even lattice of rank $g$ with $A^\sharp/A \simeq (\Z/p\Z)^m$.
The isomorphism class of an $m$-dimensional linking quadratic space $V$ over $\Z/p\Z$ is determined by its Gauss sum $\gamma(V)=|V|^{-1/2} \sum_{v \in V} e^{2\pi i \, {\rm q}(v)}$.
The Milgram formula \cite[Appendix 4]{Milnor_symm_bil} asserts $\gamma({\rm res}\, A) = e^{\frac{2 \pi i g}{8}} = \gamma({\rm res}\, {\rm Q}_g)$ and proves ${\rm res}\, A\, \simeq {\rm res}\, {\rm Q}_g$. \ps
The even lattices $L$ containing $A$ with index $p^i$ are in natural bijection with the totally isotropic subspaces of dimension $i$ over $\Z/p\Z$ inside ${\rm res}\, A$, via the map $L \mapsto L/A$.
We have already proved ${\rm res}\, A \,\simeq \,{\rm H}(\Z/p\Z)^2 \,\oplus \,{\rm res}\, {\rm E}_g$.
By Witt's theorem, any isotropic line (or plane) is thus part of a totally isotropic flag of ${\rm res}\, A$.
By Proposition \ref{prop:uniq_sub_E}, it only remains to show that any even lattice $L$ containing $A$ with $\dim_{\Z/p\Z} \,L/A \,=\, 2$  is isometric to ${\rm E}_g$.
But such an $L$ has determinant $1$ in the case $g=8$, and determinant $3$ otherwise.
As is well known, this shows $L \simeq {\rm E}_8$ in the first case, and $L \simeq {\rm E}_6$ in the second
(use e.g. that such a lattice must be the orthogonal of an ${\rm A}_2$ embedded in ${\rm E}_8$).
\end{pf}

This proposition implies in particular that the fixed point lattice $Q$ considered in Lemma \ref{lemmedefq}, in the case of an element $c$ with shape $1^4 5^4$, is isometric to ${\rm Q}_8$.

\begin{prop} \label{prop:aut_Qg}
For $g=6,8$, the natural morphism $\mathrm{O}({\rm Q}_g) \to \mathrm{O}(\res {\rm Q}_g)$ is an isomorphism.
\end{prop}

\begin{pf}  Set $A={\rm Q}_g$.
Fix an isotropic line $D$ in the quadratic space $\res A$ over $\mathbb{F}_p$ (with $p=3$ for $g=6$, $p=5$ otherwise).
We have a canonical filtration $0 \subset D \subset D^\perp \subset {\rm res}\, A$, and a nondegenerate quadratic space $V= D^\perp/D$ over $\mathbb{F}_p$.
The stabilizer $P$ of $D$ in ${\rm O}({\rm res}\, A)$ is in a natural
(splittable) exact sequence
\begin{equation} \label{parabolicP}
  1 \longrightarrow \, U \longrightarrow P \longrightarrow  {\rm
  GL}(D) \times {\rm O}(V) \longrightarrow 1
\end{equation}
($P$ is a ``parabolic subgroup'' with ``unipotent radical'' $U$).
We have an isomorphism $\beta: U \isomo \Hom((\res A) / D^\perp, V)$
characterized by
\[ g(x) \equiv x + \beta(g)(x) \mod D \]
for all $g \in U$ and $x \in \res A$.
(By duality $U$ is also naturally isomorphic to $\Hom(V,D)$, but we will not
need this point of view.)
Denote by $B$ the even lattice defined as the inverse image of $D$ in $A^\sharp$.
We have natural isomorphisms $V \simeq  {\rm res}\, B$ and $B/A \simeq D$ (see
Lemma \ref{lemm:sub_sup_latt_isot} (i)).
The stabilizer $S$ of $D$ in ${\rm O}(A)$ is ${\rm O}(A) \cap {\rm O}(B)$.
By Proposition \ref{coreg}, we are left to check that the natural map $S \rightarrow P$ is an isomorphism.
We first study ${\rm O}(A) \cap {\rm O}(B)$. Set $k=g/(p-1)$. By the same proposition, we may also assume that we have
$$B={\rm A}_{p-1}^k \hspace{.5cm}\text{and} \hspace{.5cm} A\,=\,\{\,(a_i)_{1 \leq i \leq k} \,\in \, B\,\, \,\,|\,\,\,\,\sum_{i=1}^k \rho' \cdot a_i \equiv 0 \bmod p\},$$
where $\rho'$ is some Weyl vector in ${\rm A}_{p-1}$ (e.g. the vector $((p-1)/2,...,-(p-1)/2)$). Let $R =\, k\, {\bf A}_{p-1}$ be the root system of $B$.
For general reasons, the subgroup ${\rm G}(R)$ of ${\rm A}(R)$ fixing the Weyl vector $\rho=(\rho',\dots,\rho')$ of $R$ is naturally isomorphic to $\{\pm 1\}^k \rtimes \mathfrak{S}_k$ (automorphisms of the Dynkin diagram of $R$), and we have ${\rm O}(B) = {\rm A}(R) = {\rm W}(R) \rtimes {\rm G}(R)$. This proves
 $${\rm O}(B)  \simeq \mathfrak{S}_p^k \rtimes ( \{\pm 1\}^k \rtimes \mathfrak{S}_k).$$
We trivially have ${\rm G}(R) \subset {\rm O}(A)$, hence we only have to
determine ${\rm W}(R) \cap {\rm O}(A)$.
By definition of $A$, this is the subgroup of ${\rm W}(R)$ preserving $\Z \rho + p {\rm P}(R)$.
As $\rho$ is in ${\rm Q}(R)$ and $p {\rm P}(R) \subset {\rm Q}(R)$, ${\rm W}(R)
\cap {\rm O}(A)$ is also the subgroup of ${\rm W}(R)$ preserving the subspace of
the quadratic space ${\rm Q}(R) \otimes \mathbb{F}_p$ generated by $\rho$ and
its kernel $p {\rm P}(R)/ p {\rm Q}(R)$.
But the kernel of ${\rm A}_{p-1} \otimes \mathbb{F}_p$ is generated by the image $e$ of the vector $(1-p,1,\dots,1)$,
and is fixed by $\mathfrak{S}_p$.
So ${\rm W}(R) \cap {\rm O}(A)$ is the subgroup of $(\sigma_1,\dots,\sigma_k)$ in $\mathfrak{S}_p^k$
such that there is $\lambda$ in $\mathbb{F}_p^\times$ such that for all $j=1,\dots,k$ there is $b_j$ in $\mathbb{F}_p$ with
\begin{equation} \label{actionrhosurp} \sigma_j (\rho') \,\equiv \,\lambda \,\rho' \,+ \,b_j\,e \bmod\, p{\rm A}_{p-1}.\end{equation}
To go further it will be convenient to identify ${\rm A}_{p-1}$ with the subgroup of $(x_i)_{i \in \mathbb{F}_p}$ in $\Z^{\mathbb{F}_p}$ satisfying $\sum_i x_i = 0$
in such a way that we have $\rho'_i = i$ for all $i$ in $\mathbb{F}_p$.
If we do so, ${\rm W}(R) \cap {\rm O}(A)$ becomes the subgroup of $(\sigma_1,\dots,\sigma_k)$ in $\mathfrak{S}_{\mathbb{F}_p}^k$
such that there is $\lambda$ in $\mathbb{F}_p^\times$ and $b_1,\dots,b_k$ in $\mathbb{F}_p$
with $\sigma_j^{-1}(i)=\lambda i + b_j$ for all $i$ in $\mathbb{F}_p$ and all $j=1,\dots,k$ (``$k$ affine transformations with common slope'').
We have shown ${\rm W}(R) \cap {\rm O}(A) \, = \, {\mathbb{F}_p}^k \rtimes \mathbb{F}_p^\times$ and
\begin{equation}\label{stabs} {\rm O}(A) \cap {\rm O}(B) \,= \,{\mathbb{F}_p}^k \rtimes (\mathbb{F}_p^\times \times (\{ \pm 1\}^k \rtimes \mathfrak{S}_k)).\end{equation}
It remains to identify the action of this group on $\res A$.
The reduction modulo $A$ of the natural inclusions $A \subset B \subset B^\sharp
\subset A^\sharp$, is $0 \subset D \subset D^\bot \subset {\rm res}\, A$ by
definition, and we have set $V={\rm res}\, B$.
Note that ${\rm res}\, A$ is generated by $D^\perp$ and the image of the vector
$p^{-1} \rho$, and that $W(R)$ acts trivially on $V$.
Dividing Formula \eqref{actionrhosurp} by $p$ gives the action of ${\rm W}(R)
\cap {\rm O}(A)$ on $(\res A) / D$:
\begin{itemize}
  \item $\lambda$ is an element of ${\rm GL}((\res A) / D^\perp)$, which is
    naturally isomorphic to ${\rm GL}(D)$ by duality, and
  \item for $\lambda=1$, i.e.\ when considering an element of ${\rm W}(R) \cap
    {\rm O}(A)$ mapping to $U$, the family $(b_j)_{1 \leq j \leq k}$ is the
    matrix of an element of $\Hom((\res A) / D^\perp, V)$ in the bases $p^{-1}
    \rho$ of $(\res A) / D^\perp$ and $((p^{-1} e, 0, \dots, 0), \dots, (0,
    \dots, 0, p^{-1} e))$ of $V$.
\end{itemize}
The natural map $\O(A) \to \O(\res A)$ thus identifies ${\rm W}(R) \cap {\rm
O}(A)$ with the inverse image of ${\rm GL}(D) \times 1$ in $P$.
In order to conclude that ${\rm O}(A) \cap {\rm O}(B) \rightarrow P$ is an
isomorphism, we are left to check that the natural map $$\{\pm 1\}^k \rtimes
\mathfrak{S}_k \to \O(\res \,{\rm A}_{p-1}^k) = {\rm O}( {\rm I}_k \otimes
\F_p)$$ is an isomorphism.
Injectivity is clear (for any $p>2$ and $k>0$).
Surjectivity is particular to the two cases at hand: for $(p,k)=(3,3)$ or
$(5,2)$, the only elements of norm $1$ in $\mathrm{I}_k \otimes \Fp$ are the
standard basis elements and their opposites.
\end{pf}

\begin{prop}\label{q8orientable} The lattice ${\rm Q}_8$ is orientable, whereas ${\rm Q}_6$ is not.
\end{prop}

\begin{pf}
  Set again $A={\rm Q}_g$ and $p=3$ (case $g=6$) or $p=5$ (case $g=8$).
  We will view the linking quadratic space ${\rm res}\, A$ over $\Z/p\Z$ as
  traditional quadratic space over $\Z/p\Z$ by multiplying its quadratic form by
  $p$ (making it $\Z/p\Z$-valued instead of $\frac{1}{p}\Z/\Z$-valued).
  This quadratic space is nondegenerate and isotropic (it has dimension $>2$) so
  by a classical theorem of Eichler \cite[Ch.\ II \S
  8.I]{Dieudonne_geom_gpes_class} the determinant and spinor norm maps induce an
  isomorphism
  \begin{equation} \label{eichler}
    \O(\res A)^\mathrm{ab} \simeq \{\pm 1\} \times (\Fp^\times \otimes \Z/2\Z).
  \end{equation}
  We will give two elements $\gamma,\gamma'$ of ${\rm O}(A)$ inducing orthogonal
  reflections of ${\rm res}\,A$ and with distinct spinor norms.
  The previous proposition and \eqref{eichler} will then imply that $\gamma$ and
  $\gamma'$ generate ${\rm O}(A)^{\rm ab}$.\par

  Set $k=g/(p-1)$.
  By definition, $A$ is the index $p$ subgroup of the root lattice $B={\rm
  A}_{p-1}^k$ defined by $x \cdot \rho \equiv 0 \bmod p$, where
  $\rho=(\rho',\dots,\rho')$ is a fixed Weyl vector in $B$.
  As already seen in the proof of Proposition \ref{prop:aut_Qg} the subgroup $G$
  of ${\rm O}(B)$ fixing $\rho$ is a subgroup of ${\rm O}(A)$ naturally
  isomorphic to $\{\pm 1\}^k \rtimes \frak{S}_k$.
  The subgroup $1 \rtimes \mathfrak{S}_k \subset G$ is the obvious one, but the
  element $(-1,1,\dots,1) \rtimes 1$ acts on $B$ as $(x_1,\dots,x_k) \mapsto (-
  \sigma x_1,x_2,\dots,x_k)$, where $\sigma$ in $\mathfrak{S}_p$ is the unique
  element sending $\rho'$ to $-\rho'$.
  We take $\gamma,\gamma'$ in $G$ with $\gamma=(-1,1,\dots,1) \rtimes
  \mathrm{id}$ and $\gamma' = (1,\dots,1) \rtimes \tau$, where $\tau$ is a
  transposition in $\mathfrak{S}_k$.
  Then $\gamma$ and $\gamma'$ act trivially on $A^\sharp/B^\sharp \,= \,\langle
  \,p^{-1} \rho \,\rangle$ and induce orthogonal reflections of ${\rm res}\, B$
  and ${\rm res}\, A$, with spinor norm $\frac{1}{2}\det(\res {\rm A}_{p-1})$
  for $\gamma$ and $2 \cdot \frac{1}{2}\det(\res {\rm A}_{p-1})$ for $\gamma'$.
  We actually have $\frac{1}{2} \det(\res {\rm A}_{p-1}) \equiv \frac{p-1}{2}$
  in $(\Z/p\Z)^\times$, but what only matters for this proof is that these
  spinor norms are distinct, as $2$ is not a square in $(\Z/p\Z)^\times$ for
  $p=3,5$. \par
  We have $\det \gamma_{|A} = (-1)^{(p-1)/2}$ and $\det \gamma'_{|A} =
  (-1)^{p-1}= 1$: this shows that $\det$ is trivial on ${\rm O}({\rm A})$ for
  $p=5$ but not for $p=3$.
\end{pf}

For $g=6,8$, we have seen that there is a unique ${\rm O}({\rm Q}_g)$-orbit of overlattices $E \supset {\rm Q}_g$ isomorphic to ${\rm E}_g$.  We now define ${\rm Q}_{2g}$ by a doubling process.

\begin{defi} \label{defqg}
  Set $(g,p)=(6,3)$ or $(8,5)$ and fix an embedding ${\rm Q}_g \subset {\rm
  E}_g$ arbitrarily.
  Define ${\rm Q}_{2g}$ as the sublattice of ${\rm E}_g \oplus {\rm E}_g$
  consisting of elements $(x,y)$ satisfying $x + y \in {\rm Q}_g$.
  Then ${\rm Q}_{2g}$ is an even lattice, without roots, satisfying ${\rm res}\,
  {\rm Q}_{2g}\, \simeq \,{\rm H}(\Z/p\Z)^2\, \oplus \,{\rm res}\, {\rm
  E}_g^2$.\ps
\end{defi}

Let us check the last assertion in the definition above.
Note that a root in ${\rm E}_g \oplus {\rm E}_g$ must belong either to ${\rm E}_g \oplus 0$ or to $0 \oplus {\rm E}_g$,
so the fact that ${\rm Q}_g$ has no root implies that ${\rm Q}_{2g}$ has no root either.
The assertion on the residue of ${\rm Q}_{2g}$ follows from $({\rm E}_g \oplus {\rm E}_g)/{\rm Q}_{2g} \simeq (\Z/p\Z)^2$, the fact that ${\rm res}\, {\rm Q}_{2g}$ is a subquotient of the $\Z/p\Z$-vectorspace ${\rm res}\, {\rm Q}_g \oplus {\rm res}\, {\rm Q}_g$,  and Lemma \ref{lemm:sub_sup_latt_isot}.
The following statements are analogues of Propositions \ref{prop:uniq_sub_E} and
\ref{coreg} (although their proofs are slightly different). \ps

\begin{prop} \label{EplusE}
  Set $(g,p)=(6,3)$ or $(8,5)$ and $E={\rm E}_g$.
  Up to the action of $\O( E) \times \O(E)$ there is a unique sublattice $A$ of
  index $p^2$ in $E \oplus E$ without roots.
  For such an $A$, the natural map ${\rm O}(A) \cap ({\rm O}(E) \times {\rm
  O}(E)) \rightarrow {\rm GL}((E \oplus E)/A)$ is surjective.
\end{prop}
\begin{pf}
Fix $A$ as in the statement.
The sublattice $A \cap (E \oplus 0)$ of $E \oplus 0$ has
  index dividing $p^2$ and has no root, so by Proposition \ref{prop:uniq_sub_E}
  it has index $p^2$ and there is $\gamma$ in ${\rm O}(E)$ with $(\gamma \times 1)(A \cap (E \oplus 0)) = {\rm Q}_g \oplus 0$.
  Arguing similarly with $A \cap (0 \oplus E)$, we obtain the existence of $h$ in ${\rm O}(E) \times {\rm O}(E)$ such that $h(A)$ contains ${\rm Q}_g \oplus {\rm Q}_g$.
 Set $A'=h(A)$.\par
 Denote by $P$ the totally isotropic plane $E/{\rm Q}_g$ of ${\rm res}\, {\rm Q}_g$,
 and by $I$ the plane $A'/({\rm Q}_{g} \oplus {\rm Q}_g)$ inside $P \oplus P$.
 We have seen that the two natural projections $I \rightarrow P$ are injective, hence bijective.
 There is thus an element $\varphi$ in ${\rm GL}(P)$ with $I=\{(x,\varphi(x)), x \in P\}$.
 Set $S={\rm O}(E) \cap {\rm O}({\rm Q}_g)$.
 By Proposition \ref{prop:aut_Qg}, the natural morphism $S \rightarrow {\rm GL}(P)$ is surjective.
 By multiplying $h$ by a suitable element in $1 \times {\rm O}(E)$ we may thus assume that we have $\varphi = -{\rm id}_P$, that is, $A' = {\rm Q}_{2g}$.
 We have proved the first assertion.
 For the second, observe that $S$ embeds diagonally in ${\rm O}(E) \times {\rm O}(E)$, and as such, it preserves ${\rm Q}_{2g}$ and acts on the totally isotropic plane
 $P'=(E \oplus E)/{\rm Q}_{2g}$ of ${\rm res}\, {\rm Q}_{2g}$.
 Moreover, the natural map $$E/{\rm Q}_g \rightarrow (E \oplus E)/{\rm Q}_{2g},\,\,\, x \mapsto (x,0) \bmod {\rm Q}_{2g},$$
 defines an $S$-equivariant isomorphism $P \rightarrow P'$.
 The surjectivity of $S \rightarrow {\rm GL}(P)$ thus implies that of $S
 \rightarrow {\rm GL}(P')$.
  \end{pf}

\begin{prop} \label{coregdoubl} Let $(g,p,m)$ be either $(8,5,4)$ or $(6,3,6)$.
Up to isometry, ${\rm Q}_{2g}$ is the unique even lattice of rank $2g$ without roots satisfying ${\rm Q}_{2g}^\sharp /{\rm Q}_{2g} \simeq (\Z/p\Z)^m$. \par
Moreover, ${\rm O}({\rm Q}_{2g})$ permutes transitively the totally isotropic planes $($resp. lines, resp. flags$)$ of ${\rm res}\, {\rm Q}_{2g}$.
The inverse image in ${\rm Q}_{2g}^\sharp$ of such an isotropic plane (resp. line) is isometric to ${\rm E}_g \oplus {\rm E}_g$ (resp. to an even lattice with root system $m {\bf A}_{p-1}$).
\end{prop}

\begin{pf} Let $A$ be an even lattice of rank $2g$ with $A^\sharp/A \simeq (\Z/p\Z)^m$.
The Milgram formula applied to $A$ and ${\rm Q}_{2g}$ shows ${\rm res} \,A\, \simeq \,{\rm res}\, {\rm Q}_{2g}$ (see the proof of Corollary \ref{coreg}).
The even lattices $L$ containing $A$ with index $p^i$ are in natural bijection with the totally isotropic subspaces of dimension $i$ over $\Z/p\Z$ inside ${\rm res}\, A$, via the map $L \mapsto L/A$.
As we have ${\rm res}\, A \,\simeq {\rm H}(\Z/p\Z)^2 \oplus {\rm res}\, {\rm E}_g^{2}$, the maximal isotropic subspaces of ${\rm res}\, A$ have dimension $2$ over $\Z/p\Z$.
Fix such a plane in ${\rm res}\, A$ and denote by $F$ its inverse image in $A^\sharp$.
We have ${\rm res}\, F \simeq {\rm res}\, {\rm E}_g^2$, so $F$ is an even lattice with same rank and residue as ${\rm E}_g \oplus {\rm E}_g$. \ps
Assume first $g=8$. Then $F$ is unimodular. We know since Witt that it is either isometric to ${\rm E}_8 \oplus {\rm E}_8$, or to a certain lattice ${\rm E}_{16}$ with root system ${\bf D}_{16}$.
Assuming furthermore that $A$ has no root, we claim that the $F$ cannot be isometric to ${\rm E}_{16}$.
Indeed, using the method explained in the proof of Proposition \ref{prop:uniq_sub_E}, Lemma \ref{lemm:orb_root_latt} (i) and an inspection of the Dynkin diagram of ${\bf D}_{16}$ (including the $n_i$'s):
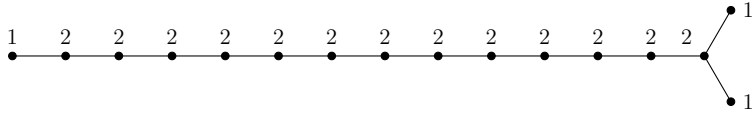
\begin{center}
\dynkin[upside down,labels={1,2,2,2,2,2,2,2,2,2,2,2,2,2,1,1},edge length=.7cm]{D}{16}
\end{center}
show that an index $5$ subgroup of ${\rm D}_{16}$ always contains an irreducible root system isomorphic to ${\bf A}_5$.
But ${\bf A}_5$ has Coxeter number $6$, so ${\rm A}_5$ has no index $5$ subgroup without roots by Lemma \ref{lemm:orb_root_latt} (ii): this proves the claim.\ps
Assume now $g=6$. We have ${\rm res}\, F\, \simeq \,- \,{\rm res} \,{\rm A}_2^2$.
This is well-known to imply that $F$ is isometric to $\rmE_6 \oplus \rmE_6$, $\rmE_8 \oplus
  \mathrm{A}_2 \oplus \mathrm{A}_2$ or to a certain lattice ${\rm E}_{12}$ having root system ${\bf D}_{10}$.
  (One way to prove this is to start by observing that such a lattice is the
  orthogonal of some ${\rm A}_2 \oplus {\rm A}_2$ embedded an even unimodular
  lattice, hence in ${\rm E}_8 \oplus {\rm E}_8$ or in ${\rm E}_{16}$.)
   An inspection of the Dynkin diagrams of ${\bf E}_8$ and ${\bf D}_{10}$ shows
   that an index $3$ subgroup of ${\rm E}_8$ or ${\rm D}_{10}$ always contains an irreducible root system isomorphic to ${\bf A}_3$,
   whose Coxeter number is $>3$.
   Assuming $A$ has no root, this implies $F \simeq {\rm E}_6 \oplus {\rm E}_6$ by  Lemma \ref{lemm:orb_root_latt} (ii). \ps
  We have just shown that in both cases, assuming $A$ has no roots, the inverse image in $A^\sharp$ of a totally isotropic plane of ${\rm res}\, A$ is isometric to ${\rm E}_g \oplus {\rm E}_g$.
  By Proposition \ref{EplusE}, there is a unique isometry class of pairs $(A,F)$ with $F \simeq {\rm E}_g \oplus {\rm E}_g$, $A$ of index $p^2$ in $F$, and ${\rm R}(A)= \emptyset$.
  This shows $A \simeq {\rm Q}_{2g}$ as well as the transitivity of ${\rm O}({\rm Q}_{2g})$ on the totally isotropic planes in ${\rm res}\, {\rm  Q}_{2g}$.
  Moreover, the same proposition also asserts that the stabilizer in ${\rm O}({\rm Q}_{2g})$ of an isotropic plane $P$ in ${\rm res}\, {\rm Q}_{2g}$ surjects naturally onto ${\rm GL}(P)$.
  This shows the transitivity of ${\rm O}({\rm Q}_{2g})$ on the isotropic lines (resp. flags) in ${\rm res}\, {\rm Q}_{2g}$. \ps
  Fix an even lattice $B \subset {\rm E}_g$ containing ${\rm Q}_g$ with index $p$.
  We known from Proposition \ref{prop:uniq_sub_E} that such a $B$ exists and is a root lattice with root system $\frac{g}{p-1} {\bf A}_{p-1}$.
  The sublattice $C \subset {\rm E}_g \oplus {\rm E}_g$ whose elements $(x,y)$ satisfy $x+y \in B$ contains ${\rm Q}_{2g}$,
  and defines an isotropic line $C/{\rm Q}_{2g}$ in ${\rm res}\, {\rm Q}_{2g}$.
  Its root system ${\rm R}(C)$ is isomorphic to $2\frac{g}{p-1} {\bf A}_{p-1}$.
  This concludes the proof of the proposition.
\end{pf}

\begin{prop} \label{prop:doubling}
  For $g=6,8$, the natural morphism $\O({\rm Q}_{2g}) \to \O(\res \, {\rm
  Q}_{2g})$ is surjective, with kernel isomorphic to $\Z/3\Z$ for $g=6$, $\Z/5\Z
  \rtimes \Z/2\Z$ for $g=8$.
\end{prop}
\begin{proof} Set $E={\rm E}_g$,  $Q={\rm Q}_g$  and consider the group $S={\rm O}(E) \cap {\rm O}(Q)$. The inclusions $Q \subset E \subset E^\sharp \subset Q^\sharp$ define a  composition series $U_1 \lhd U_2 \lhd U_3 \lhd S$, where:
  \begin{itemize}
  \item $U_3$ is the kernel of the natural morphism $\beta : S \rightarrow {\rm GL}(E/Q)$,\ps
  \item $U_2$ is the kernel of the natural morphism $\beta_3 : U_3 \rightarrow {\rm O}({\rm res}\, E)$, \ps
  \item $U_1$ is the kernel of the natural morphism $\beta_2 : U_2 \rightarrow  \Hom(E^\sharp/E,E/Q)$ given by $g(x) = x + \beta_2(g)(\ol{x})$ for all $g$ in $U_2$ and all $x$ in $E^\sharp$ with image $\ol{x}$ in $E^\sharp/E$.
  \end{itemize}
Moreover, if $\Hom(Q^\sharp/E^\sharp,   E/Q)^{\rm antisym}$ denotes the group of antisymmetric group homomorphisms $Q^\sharp/E^\sharp \rightarrow E/Q$,
with $E/Q$ identified with ${\rm Hom}(Q^\sharp/E^\sharp,\Q/\Z)$ using the symmetric bilinear form of $Q\otimes \Q$, we have a natural morphism: \begin{itemize}
\item $\beta_1 :  U_1 \to \Hom(Q^\sharp/E^\sharp,   E/Q)$, given by $g(x) = x + \beta_1(g)(\ol{x})$ for all $g$ in $U_1$ and all $x $ in $Q^\sharp/Q$ with image $\ol{x}$ in $Q^\sharp/E^\sharp$.
\end{itemize}
\noindent Last but not least, since the natural map $\O(Q) \to \O(\res \,Q)$ is an isomorphism by Proposition \ref{prop:aut_Qg}, the morphisms $\beta$, $\beta_3$, $\beta_2$ above are surjective, and $\beta_1$ is an isomorphism. Let $p$ denote the prime such that we have $E/Q \simeq (\Z/p\Z)^2$, we have proved in particular $U_1 \simeq \Z/p\Z$. \ps

Set now $F = E \oplus E$, $A = {\rm Q}_{2g}$ and consider the group $T = {\rm O}(F) \cap {\rm O}(A)$.  On the one hand, we have $\O(F) = \O(E)^2 \rtimes \mathfrak{S}_2$. As
  $\mathfrak{S}_2$ clearly stabilizes $A$, this shows
   \[ T = G \rtimes \mathfrak{S}_2 \hspace{.5 cm} \text{with} \hspace{,5 cm} G = \left\{ (g_1,g_2) \in S \times S \,\middle|\,
  \beta(g_1) = \beta(g_2) \right\}.\]
On the other hand, $T$ is also the stabilizer in ${\rm O}(A)$ of the subspace $F/A$ of ${\rm res}\, A$.
By Proposition \ref{coregdoubl} we are left to prove that the natural morphism $\nu : T \rightarrow \overline{T}$,
where $\overline{T}$ is the stabilizer of $F/A$ in ${\rm O}({\rm res} A)$, is a surjection whose kernel is as in the statement. \ps

  To the inclusions $A \subset F \subset F^\sharp \subset A^\sharp$ is
  associated as above a composition series $V_1 \lhd V_2 \lhd V_3 \lhd
  \overline{T}$, whose successive quotients $\overline{T}/V_3$, $V_3/V_2$,
  $V_2/V_1$ and $V_1$ are naturally identified with the groups $\GL(F/A)$, ${\rm
  O}( {\rm res}\, F)$, $\Hom(F^\sharp/F,F/A)$ and
  $\Hom(A^\sharp/F^\sharp,F/A)^{\rm antisym}$.
  The following observations below will prove $\nu(U_i \times U_i \rtimes \mathfrak{S}_2) = V_i$ for $i=1,2,3$, $\nu(T) =\overline{T}$ and identify ${\rm ker}\, \nu$.\ps

  {\it Action of $\mathfrak{S}_2$.}
  By definition of $A$ the group $\mathfrak{S}_2$ acts trivially on $F/A$, hence
  on $A^\sharp/F^\sharp$ as well.
  Moreover, it swaps the two factors of ${\rm res}\, F \,=\, {\rm res}\, E
  \oplus {\rm res}\, E$.
  Recall that this linking quadratic space is $0$ for $g=8$, isomorphic to ${\rm
  H}(\Z/3\Z)$ for $g=6$.
  As $V_2$ is a $p$-group and $p$ is odd it follows that $\mathfrak{S}_2$ acts
  trivially on ${\rm res}\, A$ for $g=8$. \ps

  {\it Action of $G$ on $F/A$}. As $\mathfrak{S}_2$ acts trivially on $F/A$, Proposition \ref{EplusE} implies that $\nu$ induces an isomorphism $T/((U_3 \times U_3) \rtimes \mathfrak{S}_2)  \simeq  \overline{T}/V_3$.\ps

  {\it Restriction of $\nu$ to $(U_3 \times U_3) \rtimes \mathfrak{S}_2$}.
  For $g=(g_1,g_2)$ in $U_3 \times U_3$ and $(x_1,x_2)$ in $F^\sharp=E^\sharp
  \oplus E^\sharp$, we have $g(x_1,x_2) \equiv (g_1(x_1), g_2(x_2)) \bmod F$.
  So $\nu$ induces an isomorphism between $U_3/U_2 \times U_3/U_2$ and the
  subgroup $\O(\res E) \times \O(\res E)$ of $\O(\res\, E\, \oplus\, {\rm res}
  \,E)$.
  This subgroup has index $2$ for $g=6$ (and $1$ for $g=8$) but recall that in
  this case $\mathfrak{S}_2$ swaps the two factors of ${\rm res}\, E \oplus {\rm
  res}\, E$.\ps

  {\it Restriction of $\nu$ to $U_2 \times U_2$}.
  The map $\iota : E/Q \rightarrow F/A, x \mapsto (x,0),$ is an isomorphism.
  For $g=(g_1,g_2)$ in $U_2 \times U_2$ and $(x_1,x_2)$ in $F^\sharp=E^\sharp
  \oplus E^\sharp$ we thus have the following equalities in $F/A$ :
  \[ g(x_1,x_2) - (x_1,x_2) = (\beta_2(g_1)(x_1), \beta_2(g_2)(x_2)) =
  \iota(\beta_2(g_1)(x_1)+\beta_2(g_2)(x_2)). \]
  This shows that $\nu$ induces an isomorphism $U_2/U_1 \times U_2/U_1 \isomo
  V_2/V_1$. \ps

  {\it Restriction of $\nu$ to $U_1 \times U_1$}.
  We have $A^\sharp = \{ (y_1,y_2) \in Q^\sharp \oplus Q^\sharp\, |\, y_1 \equiv
  y_2 \bmod E^\sharp\}$.
  For $g = (g_1,g_2)$ in $U_1 \times U_1$ and $(x_1+y,x_2+y)$ in
  $A^\sharp$ with $x_i$ in $E^\sharp$ and $y$ in $Q^\sharp$, we have the
  following equality in $F/A$ (with $\iota$ defined as above):
  \[ g(x_1+y,x_2+y) - (x_1+y,x_2+y) = (\beta_1(g_1)(y), \beta_1(g_2)(y)) =
  \iota(\beta_1(g_1)(y)+\beta_1(g_2)(y)).  \]
  This shows $\nu(U_1 \times U_1)=V_1$ and
  \begin{equation} \label{kernuu1}
    {\rm ker}\, \nu_{| U_1 \times U_1}= \{(g_1,g_2) \in U_1 \times U_1 \,|\,
    \beta_1(g_1) + \beta_1(g_2) = 0 \}\simeq \Z/p\Z.
  \end{equation}
  \ps

\noindent All in all, we have shown $\nu(T)=\overline{T}$, and if $K \simeq \Z/p\Z$ denotes the group in \eqref{kernuu1}, $\ker \nu = K \rtimes \mathfrak{S}_2$ for $g=8$, and $\ker \nu = K$ for $g=6$.
\end{proof}

\begin{prop}
  The lattice ${\rm Q}_{12}$ is orientable.
\end{prop}
\begin{proof}
  As the kernel of
  ${\rm O}({\rm Q}_{12}) \rightarrow {\rm O}({\rm res}\, {\rm Q}_{12})$
  has odd cardinality (namely $3$) by Proposition
  \ref{prop:doubling}, it is contained in $\SO({\rm Q}_{12})$.
  Arguing as in the proof of Proposition \ref{q8orientable}, we are left to find
  two elements $g, g'$ in ${\rm O}({\rm Q}_{12})$ with determinant $1$ and whose
  images in ${\rm O}({\rm res}\, {\rm Q}_{12})$ are reflections with distinct
  spinor norms.
  In the following arguments, it will be convenient to view linking quadratic
  spaces over $\Z/3\Z$ as traditional quadratic spaces over $\Z/3\Z$ by
  multiplying their quadratic form by $3$ (which becomes then $\Z/3\Z$-valued
  instead of $\frac{1}{3}\Z/\Z$-valued), so that it makes sense to talk about
  their determinant.
  \par

  Consider first the non-trivial element $g$ of the group $\mathfrak{S}_2$
  naturally acting on ${\rm E}_6 \oplus {\rm E}_6$.
  Then $g$ acts trivially on $({\rm E}_6 \oplus {\rm E}_6)/{\rm Q}_{12}$, and
  in the obvious way on ${\rm res} \,{\rm E}_6 \,\oplus\, {\rm res} \, {\rm
  E}_6$.
  It acts thus as a reflection with spinor norm $2 \cdot \frac{1}{2}\det ({\rm
  res}\, {\rm E}_6)$ in $(\Z/3\Z)^\times$ (the squares of $(\Z/3\Z)^\times$ are
  $\{1\}$).
  Moreover, we have $\det g = (-1)^6=1$. \par

  Let $s$ be an order two element of ${\rm O}({\rm Q}_6) \cap {\rm O}({\rm E}_6)$ acting trivially on
  ${\rm E}_6/{\rm Q}_6$ and by $-1$ on ${\rm res}\, {\rm E}_6$.
  Such an $s$ exists by Proposition \ref{prop:aut_Qg}.
  By construction, it is a reflection in ${\rm O}({\rm res}\, {\rm Q}_6)$ with
  spinor norm $\frac{1}{2} \det ({\rm res}\, {\rm E}_6)\,=  \det ({\rm res}\,
  {\rm A}_2)$ in $(\Z/3\Z)^\times$.
  So $s$ is conjugate in ${\rm O}({\rm Q}_6)$ to the element denoted $\gamma'$
  in the proof of Proposition \ref{q8orientable}, and we have thus $\det s =
  \det \gamma' =1$ as was shown {\it loc. cit.}
  Consider now the order $2$ element $g' = (s,1)$ in ${\rm O}({\rm E}_6) \times {\rm O}({\rm E}_6)$.
  As $s$ preserves ${\rm Q}_6$ ands acts trivially on ${\rm E}_6/{\rm Q}_6$, the element $g'$ preserves ${\rm Q}_{12}$
  and has a trivial image in ${\rm
  GL}(({\rm E}_6 \oplus {\rm E}_6)/{\rm Q}_{12})$. It acts as ${\rm diag}(-1,1)$
  in ${\rm res} \,{\rm E}_6 \,\oplus\, {\rm res} \, {\rm E}_6$.
  It acts thus on ${\rm res}\, {\rm Q}_{12}$ as a reflection with spinor norm
  $\frac{1}{2}\det({\rm res}\, {\rm E}_6)$.
  This spinor norm is not the same as that of $g$ as $2$ is not a square in
  $(\Z/3\Z)^\times$.
  The orientabiliy of ${\rm Q}_{12}$ follows then from the equalities $\det g' =
  \det s \times 1 =1$.
\end{proof}

We finally set ${\rm Q}_0=0$ and ${\rm Q}_{24}= {\rm Leech}$.
We denote by ${\rm n}_g$ the number of isometric embeddings ${\rm Q}_g \rightarrow {\rm Leech}$,
and by ${\rm K}_g$ the kernel of the morphism ${\rm O}({\rm Q}_g) \rightarrow {\rm O}({\rm res}\, {\rm Q}_g)$.
By Propositions \ref{prop:aut_Qg} and \ref{prop:doubling} we have $|{\rm K}_g|=1$ for $g <12$, $|{\rm K}_{12}|=3$, $|{\rm K}_{16}|=10$, and of course ${\rm K}_{24}={\rm O}({\rm Leech})$. \ps

\begin{prop} For $g$ in $\{0, 8,12,16, 24\}$ there is a unique ${\rm O}({\rm Leech})$-orbit of sublattices $Q$ of ${\rm Leech}$ with $Q \simeq {\rm Q}_g$, and we have ${\rm n}_g \,|{\rm K}_{24-g}| = |{\rm O}({\rm Leech})|$.
\end{prop}
\begin{pf}  Let $Q$ be a sublattice of ${\rm Leech}$ isomorphic to ${\rm Q}_g$.
By  \cite[Prop. B.2.2 (d)]{CheLan}, the lattice $Q^\perp$ satisfies ${\rm res}\, Q^\perp \,\simeq\,-{\rm res}\, Q$.
By Propositions \ref{coreg} and \ref{coregdoubl}, we have $Q^\perp \simeq {\rm Q}_{g'}$ with $g'=24-g$.
Moreover, the stabilizer of ${\rm Q}$ in ${\rm O}({\rm Leech})$ trivially coincides with that of $Q^\perp$.
To prove uniqueness we are thus left to show that there is a unique ${\rm
O}({\rm Q}_g) \times {\rm O}({\rm Q}_{g'})$-orbit of overlattices $L \supset
{\rm Q}_g \oplus {\rm Q}_{g'}$ with $L \simeq {\rm Leech}$.
Note that the existence of such an $L$ follows from Lemma \ref{lemmedefq}. \ps
Consider now an arbitrary maximal isotropic subspace $I$ in ${\rm res}\, {\rm Q}_g \oplus {\rm res}\, {\rm Q}_{g'}$
(which is a hyperbolic linking quadratic space over $\Fp$ with $p=5$ or $3$).
Let $L$ be the inverse image of $I$ in ${\rm Q}_g^\sharp \oplus {\rm
Q}_{g'}^\sharp$, an even unimodular lattice.
We assume furthermore that it has no root.
Then $L \cap ({\rm Q}_g^\sharp \oplus 0)$ is an even lattice without root containing ${\rm Q}_g \oplus 0$.
By Propositions \ref{coreg} and \ref{coregdoubl} it must be ${\rm Q}_g \oplus 0$, and similarly we have $L \cap (0 \oplus {\rm Q}_{g'}^\sharp)= 0 \oplus {\rm Q}_{g'}$.
It follows that both projections $I \rightarrow {\rm res}\, {\rm Q}_g$ and $I \rightarrow {\rm res}\, {\rm Q}_{g'}$ are injective, hence isomorphisms.
So there is an isometry $\varphi : {\rm res} \,{\rm Q}_g\, \isomo - {\rm res}\, {\rm Q}_{g'}$ such that we have $I = I_\varphi$,
with $I_\varphi\,=\, \{(x,\varphi(x)), x \in {\rm res}\, {\rm Q}_g\}$.
By Propositions \ref{prop:aut_Qg} and \ref{prop:doubling}, the map ${\rm O}({\rm Q}_{g'}) \rightarrow {\rm O}({\rm res}\, {\rm Q}_{g'})$ is surjective.
This shows that $1 \times {\rm O}({\rm Q}_{g'})$ permutes transitively the $I_\varphi$,
and that the stabilizer in this group of any $I_\varphi$ is the kernel of ${\rm O}({\rm Q}_{g'}) \rightarrow {\rm O}({\rm res} \,{\rm Q}_{g'})$, and we are done. \end{pf}

\noindent We have also proved above the following:

\begin{coro}\label{qgdansleech} Fix $g$ in $\{0, 8,12,16, 24\}$ and an isometric embedding of ${\rm Q}_g \oplus {\rm Q}_{g'}$ in ${\rm Leech}$, with $g'=24-g$.
The stabilizer $S$ of ${\rm Q}_g$ in ${\rm O}({\rm Leech})$ is ${\rm O}({\rm Leech}) \cap ({\rm O}({\rm Q}_g) \times {\rm O}({\rm Q}_{g'}))$ and the natural map $S \rightarrow {\rm O}({\rm Q}_g)$ is surjective with kernel $1 \times {\rm K}_{g'}$.
\end{coro}

\begin{prop}
  The lattice ${\rm Q}_{16}$ is orientable.
\end{prop}

\begin{proof}
  Fix an isometric embedding of ${\rm Q}_8 \oplus {\rm Q}_{16}$ in ${\rm
  Leech}$.
  By Corollary \ref{qgdansleech}, for any $\gamma$ in ${\rm O}({\rm Q}_{16})$
  there is $\gamma'$ in ${\rm O}({\rm Q}_8)$ such that $\gamma \oplus \gamma'$
  is in ${\rm O}({\rm Leech})$.
  As any element of ${\rm O}({\rm Leech})$ has determinant $1$ we have $\det
  \gamma \det \gamma'=1$.
  But we have $\det \gamma'=1$ as ${\rm Q}_8$ is orientable, hence  $\det \gamma
  =1$.
\end{proof}

We have used several times the following simple lemma.

\begin{lemm} \label{lemm:sub_sup_latt_isot} Let $L$ be an even lattice. \begin{itemize}
\item[(i)] The map $M \mapsto M/L$ defines a bijection between the set of even lattices $M$ in $L \otimes \Q$ containing $L$ and the set of totally isotropic subgroups $I \subset {\rm res}\, L$ $($that is, with ${\rm q}(I)=0)$. In this bijection, we have ${\rm res}\, M \simeq I^\perp/I$. If furthermore $I$ is a direct summand of the abelian group ${\rm res}\, L$, and if $|I|$ is odd, then we have a noncanonical isomorphism ${\rm res}\, L \simeq {\rm H}(I) \oplus {\rm res}\, M$.\ps
\item[(ii)] Let $h$ be an odd integer $\geq 1$ and  $x \in L$ with $x \cdot x \equiv 0 \bmod h$. Assume that the natural map $L \rightarrow \Z/h\Z, y \mapsto y \cdot x \bmod h$ is surjective, and denote by $M$ its kernel. Then $M$ is an even lattice with $L/M \simeq \Z/h\Z$ and ${\rm res}\, M \simeq {\rm res}\, L \oplus {\rm H}(\Z/h\Z)$. \ps
\end{itemize}
\end{lemm}
\begin{pf} The first two assertions in (i) are obvious \cite[Prop. 2.1.1]{CheLan}. For the last assertion of (i) choose first a subgroup $J$ of $I^\perp$ with $I^\perp = J \oplus I$. Then $J$ is nondegenerate in ${\rm res}\, L$, $I$ is a totally isotropic direct summand of $V:=J^\perp$, and we have an exact sequence $0 \rightarrow I \rightarrow V \rightarrow {\rm Hom}(I,\Q/\Z) \rightarrow 0$. We now argue as in the proof of Proposition 2.1.2 of \cite{CheLan} (beware however that the statement loc. cit. does not hold for linking quadratic spaces of even cardinality). Choose a supplement $I'$ of $I$ in $V$, {\it i.e.} $V= I \oplus I'$. As $V$ is nondegenerate, any bilinear form on $I'$ is of the form $(x,y) \mapsto x \cdot \varphi(y)$ for some morphism $\varphi : I' \rightarrow I$. We apply this to the form $(x,y) \mapsto \frac{1}{2} x \cdot y$, which is well defined as $|V|$ is odd. Then the subgroup $\{x-\varphi(x), x \in I'\}$ is a totally isotropic supplement of $I$ in $V$. This implies $V \simeq {\rm H}(I)$ (see Proposition-Definition 2.1.3 {\it loc. cit.}). \par
For assertion (ii), consider the natural map $M^\sharp \rightarrow \Z/h\Z, y \mapsto y \cdot x \bmod h$. This is well defined as we have $x \in M$ by assumtion, and its restriction to $L$ induces an isomorphism $L/M \isomo \Z/h\Z$. So $L/M$ is a direct summand of ${\rm res}\, M$ and we conclude the proof by (i).
\end{pf}


\section{Standard ${\rm L}$-functions of the eigenforms ${\rm F}_g$}\label{complements}

In this section, we show that the Siegel modular forms ${\rm F}_g$ defined in \eqref{formfg} are eigenforms and give an expression for their standard ${\rm L}$-functions. \ps

\begin{prop}\label{Qrootlatticeinvform} Let $L$ be an integral lattice whose roots generate $L \otimes \R$. For any $g\geq 1$, there is no nonzero, ${\rm O}(L)$-invariant, alternating $g$-form on $L$. \end{prop}

\begin{proof}
  Let $\omega : L^g \rightarrow \R$ be such a form.
  It is enough to show $\omega(x_1,\dots,x_g)=0$ for any $x_1,\dots,x_g$ in $L$,
  with $x_i$ roots of $L$.
  Fix such $x_i$ and let $s$ be the reflection associated to the root $x_1$.
  We have $s \in {\rm O}(L)$ as $L$ is integral,  $s(x_1)=-x_1$ and $s(x_i)
  \,\in \, x_i \,+ \,\Z \,x_1$ for all $i$, hence the following equalities
  \[ \omega(x_1,x_2,\dots,x_g)=\omega(s(x_1),s(x_2),\dots,s(x_g)) = -
  \omega(x_1,x_2,\dots,x_g) =0. \]
\end{proof}

Fix an integer $n \equiv 0 \bmod 8$ and consider the set $\mathcal{L}_n$ of even unimodular lattices in the standard Euclidean space $V=\R^n$.
For all $g\geq 1$ we denote by ${\rm Alt}_n^g$ the free $\R$-vector space with generators the $(L,\omega)$,
with $L$ in $\mathcal{L}_n$ and $\omega$ an alternating $g$-form on $V$,
and with relations the $$(\gamma^{-1}(L), \omega \circ \gamma) = (L,\omega) \, \, \, \,\text{and}\,\,\,\, (L, \lambda \,\omega + \omega') \,=\,\lambda \,(L,\omega)\, +\, (L,\omega'),$$
for all $L$ in $\mathcal{L}_n$, all $\gamma$ in ${\rm O}(V)$, all alternating $g$-forms $\omega,\omega'$ on $V$ and all $\lambda$ in $\R$.
It follows readily from these definitions that the Siegel theta series construction $(L;\omega) \mapsto \Theta(L;\omega) = \sum_{\underline{v} \in L^g} \omega(\underline{v}) q^{\frac{\underline{v}\cdot \underline{v}}{2}}$ factors through an $\R$-linear map
\begin{equation} \label{thethaalt} \Theta : {\rm Alt}_n^g \longrightarrow {\rm S}_{n/2+1}({\rm Sp}_{2g}(\Z)).\end{equation}
If $L_1,\dots,L_h$ denote representatives for the isometry classes of even unimodular lattices in $V$,
we also have an $\R$-linear isomorphism
\begin{equation} \label{dimaltng} {\rm Alt}_n^g \simeq \bigoplus_{i=1}^h (\Lambda^g V^\ast)^{{\rm O}(L_i)}.\end{equation}
The classification of even unimodular lattices in rank $\leq 24$ (or simply, Venkov's argument in \cite[Chap. 18, \S 2, Prop. 1]{splag}) shows that appart from ${\rm Leech}$ these lattices are generated over $\Q$ by their roots.
Proposition \ref{Qrootlatticeinvform} and Formula \eqref{dimaltng} thus show that ${\rm Alt}_n^g$ vanishes for $n<24$, and together with \eqref{invaltleech}, imply:

\begin{prop} ${\rm Alt}_{24}^g$ has dimension $1$ for $g$ in $\{8,12,16,24\}$, $0$ otherwise.
\end{prop}

Let us denote by ${\rm O}_n$ the orthogonal group scheme of a fixed even unimodular lattice of rank $n$,
{\it e.g.} of ${\rm D}_n \,+\, \Z \,e$ with $e\,= \,\frac{1}{2}(1,\dots,1)$. For any finite dimensional representation $U$ of ${\rm O}(V)$, the space\footnote{We have a similar definition with ${\rm O}$ replaced by ${\rm SO}$ that we will also use below.} ${\rm M}_U({\rm O}_n)$ of ${\rm O}(V)$-equivariant functions $\mathcal{L}_n \rightarrow U$, is the space of level $1$ automorphic forms of ${\rm O}_n$ with coefficients in $U$  \cite[\S 4.4.4]{CheLan}. As such it is equipped with an action of the (commutative) Hecke ring ${\rm H}({\rm O}_n)$ of ${\rm O}_n$ \cite[\S 4.2.5 \& \S 4.2.6]{CheLan}. For all $g\geq 1$, the space ${\rm Alt}_n^g$ is canonically isomorphic to the dual of ${\rm M}_{\Lambda^g V}({\rm O}_n)$, hence carries an ${\rm H}({\rm O}_n)$-action as well. As an example, for any prime $p$ the Kneser $p$-neighbor operator is the endomorphism of ${\rm Alt}_n^g$ sending $(L,\omega)$ to the sum of $(L',\omega)$ over the $L'$ in $\mathcal{L}_n$ with $L \cap L'$ of covolume $p$. The so-called {\it Eichler commutations relations} imply that the map $\Theta$ in \eqref{thethaalt} sends an ${\rm H}({\rm O}_n)$-eigenvector on the left-hand side either to $0$ or to a Siegel eigenform on the right-hand side (i.e. an ${\rm H}({\rm Sp}_{2g})$-eigenvector): see \cite{Freitag_harm_theta}, as well as \cite{Rallis_Eichler} for an interpretation in terms of Satake parameters. \ps

For $g=8,12,16,24$, the space ${\rm Alt}_{24}^g$ has dimension $1$, so it is generated by an ${\rm H}({\rm O}_{24})$-eigenvector.
Our main theorem asserts that the image of ${\rm Alt}_{24}^g$ under $\Theta$ is generated by ${\rm F}_g$ and is nonzero.
We have proved:

\begin{coro} For $g=8,12,16,24$, the Siegel modular form ${\rm F}_g$ is an eigenform. \end{coro}

We now discuss the standard ${\rm L}$-functions of the eigenforms ${\rm F}_g$, or more precisely, their collections of Satake parameters. We need some preliminary remarks and notations mostly borrowed from \cite[\S 6.4]{CheLan}.\ps

For any integer $n\geq 1$ we denote by $\mathcal{X}_n$ the set of sequences $c=(c_2,\dots,c_p,\dots,c_\infty)$, where the $c_p$ are semisimple conjugacy classes in ${\rm GL}_n(\C)$ indexed by the primes $p$, and where $c_\infty$ is a semisimple conjugacy class in ${\rm M}_n(\C)$. The direct sum and tensor product induce componentwise two natural operations $\mathcal{X}_n \times \mathcal{X}_m \rightarrow \mathcal{X}_{n+m}$ and $\mathcal{X}_n \times \mathcal{X}_m \rightarrow \mathcal{X}_{nm}$, denoted respectively $(c,c') \mapsto c \oplus c'$ and $(c,c') \mapsto c\,c'$. An important role will be played by the element $[n]$ of $\mathcal{X}_n$ such that $[n]_p$ (resp. $[n]_\infty$) has the eigenvalues $p^{\frac{n-1}{2}-i}$ (resp. $\frac{n-1}{2}-i$) for $i=0,\dots,n-1$. \ps

Any Siegel eigenform $F$ for ${\rm Sp}_{2g}(\Z)$ has an associated collection of Satake parameters, semisimple conjugacy classes in ${\rm SO}_{2g+1}(\C)$ indexed by the primes, as well as an infinitesimal character (as defined by Harish-Chandra), which may be viewed as a semisimple conjugacy class in the Lie algebra of ${\rm SO}_{2g+1}(\C)$. So $F$ gives rise to an element in $\mathcal{X}_{2g+1}$ using the natural (or ``standard'') representation ${\rm SO}_{2g+1}(\C) \rightarrow {\rm GL}_{2g+1}(\C)$. This element is called the {\it standard parameter} of $F$. Similarly, any ${\rm H}({\rm O}_n)$-eigenvector in ${\rm M}_U({\rm O}_n)$ or ${\rm M}_U({\rm SO}_n)$ gives rise to an element of $\mathcal{X}_{n}$: see e.g. \cite[Sch. 6.2.4 \& Def. 6.4.9]{CheLan}. \ps

For $g=8,12,16,24$ we denote by $\psi_g$ the standard parameter of the eigenform ${\rm F}_g$, and by $\psi'_g$ that of a generator of ${\rm M}_{\Lambda^g V}({\rm O}_{24}) = ({\rm Alt}_{24}^g)^\ast$.  By definition, $\psi_g$ is in $\mathcal{X}_{2g+1}$ and $\psi'_g$ is in $\mathcal{X}_{24}$. Using Theorem \ref{mainthm}, Rallis's aforementioned theorem asserts
\begin{equation} \label{rallis} \psi'_8 = \psi_8 \oplus [7] \, \, \, \, \, \text{and}\,\,\,\,\,\psi_g = \psi'_g \oplus [2g-23]\,\,\, \text{for}\,\,\, g \geq 12.\end{equation}
In the spirit of standard conjectures by Langlands and Arthur (see. \cite[\S 6.4.4]{CheLan}),
we will express those $\psi_g$ and $\psi'_g$ in terms of Satake parameters of certain cuspidal automorphic eigenforms for ${\rm GL}_m(\Z)$.
The four following forms will play a role:\ps

-- For $w=11,17$, we denote by $\Delta_w \in \mathcal{X}_2$ the collection of the Satake parameters, and of the infinitesimal character,
of the classical modular normalized eigenform of weight $w+1$ for ${\rm PGL}_2(\Z)$. For example, the $p$-th component of $\Delta_{11}$ has determinant $1$ and trace $\tau(p)/p^{11/2}$. The eigenvalues of $(\Delta_w)_\infty$ are $\pm \frac{w}{2}$. \ps

-- For $(w,v)=(19,7)$ and $(21,13)$, and following  \cite[\S 9.1.3]{CheLan}, there is a unique (up to scalar) cuspidal eigenform for ${\rm PGL}_4(\Z)$
whose infinitesimal character has the eigenvalues $\pm w/2,\pm v/2$:
we denote by $\Delta_{w,v}\in \mathcal{X}_4$ the collection of its Satake parameters, and of this infinitesimal character.
As explained {\it loc. cit.},  they are also the spinor parameters of generators of the $1$-dimensional space of Siegel modular forms for ${\rm Sp}_4(\Z)$
with coefficients in the representations ${\rm Sym}^6 \otimes \det^8$ and ${\rm Sym}^{12} \otimes \det^6$ of ${\rm GL}_2(\C)$ respectively.
See \cite[Tables C3 \& C.4]{CheLan} and \cite{bfgwebsite} for more information
on these Satake parameters. \ps

\begin{theo} \label{standardparamFg} The parameters $\psi_g$ and $\psi'_g$ are given by the following table:
\vspace{-.5cm}
\begin{table}[htp]
{\tiny \renewcommand{\arraystretch}{1.8} \medskip
\begin{center}
\begin{tabular}{c||c|c|c|c}
$g$ & $8$ & $12$ & $16$ & $24$ \\
\hline
$\psi_g$ & $\Delta_{21,13}[4] \oplus [1]$ & $\Delta_{19,7}[6] \oplus [1]$ & $\Delta_{17}[8] \oplus [9] \oplus [7] \oplus [1]$ & $\Delta_{11}[12] \oplus [25]$ \\
\hline
$\psi'_g$ &  $\Delta_{21,13}[4] \oplus [7] \oplus [1]$ & $ \Delta_{19,7}[6]$ & $\Delta_{17}[8] \oplus [7] \oplus [1]$ & $ \Delta_{11}[12]$ \\
\end{tabular}
\end{center}}
\end{table}
\end{theo}

\begin{pf}By Proposition 7.5.1 of \cite{CheLan}, relying on \cite{Ikeda01} and \cite{Weissauer_book} or \cite{bocherer_theta},
we have $\psi'_{24}=\Delta_{11}[12]$, and thus $\psi_{24} = \Delta_{11}[12] \oplus [25]$ by \eqref{rallis}.
The remaining parameters are harder to determine,
and at the moment we only know how to do it using Arthur's results \cite{Arthur_book} together with \cite{AMR,TaiMult}.\ps
The irreducible representation $\Lambda^{12}\, V$ of ${\rm O}(V)$ is the sum of two irreducible non-isomorphic representations $A^{\pm}$ of ${\rm SO}(V)$.
As a consequence, the two spaces ${\rm M}_{A^{\pm}}({\rm SO}_{24})$ have dimension $1$ and are isomorphic to ${\rm M}_{\Lambda^{12}\, V}({\rm O}_{24})$ as ${\rm H}({\rm O}_{24})$-modules (see \cite[\S 4.4.4]{CheLan}). The eigenvalues of $s=(\Delta_{19,7}[6])_\infty$ are $\pm i$ with $i=1,\dots,12$, so $s$ is the image in ${\rm M}_{24}(\C)$ of the infinitesimal character of $A^{\pm}$.  By Arthur's multiplicity formula for ${\rm SO}_{24}$, discussed in \cite[Thm. 8.5.8]{CheLan} and which applies by \cite{AMR,TaiMult},
there is an ${\rm H}({\rm O}_{24})$-eigenvector in ${\rm M}_{A^{\pm }}({\rm
SO}_{24})$ with standard parameter $\Delta_{19,7}[6]$: this parameter must be
$\psi'_{12}$ because we have $\dim {\rm M}_{A^\pm}({\rm SO}_{24}) = 1$. \ps

The two non-isomorphic representations $\Lambda^8\, V$ and $\Lambda^{16}\, V$ of ${\rm O}(V)$ have isomorphic and irreducible restriction $B$ to ${\rm SO}(V)$. As a consequence, the space
\begin{equation} \label{decmb} {\rm M}_B({\rm SO}_{24}) \simeq  {\rm M}_{\Lambda^8 V}({\rm O}_{24}) \oplus {\rm M}_{\Lambda^{16} V}({\rm O}_{24})\end{equation} has dimension $2$ (see \cite[\S 4.4.4]{CheLan}). Assume $\psi \in \mathcal{X}_{24}$ is either $\Delta_{21,13}[4] \oplus [7] \oplus [1]$ or $\Delta_{17}[8] \oplus [7] \oplus [1]$. The eigenvalues of $\psi_\infty$ are the $\pm i$ with $0\leq i \leq 12$ and $i \neq 4$, so $\psi_\infty$ is the image in ${\rm M}_{24}(\C)$ of the infinitesimal character of $B$. An inspection of Arthur's multiplicity formula for ${\rm SO}_{24}$ \cite[Thm. 8.5.8]{CheLan} shows that
there is an ${\rm H}({\rm O}_{24})$-eigenvector in ${\rm M}_B({\rm SO}_{24})$ with standard parameter $\psi$. These two parameters are distinct and the isomorphism \eqref{decmb} is ${\rm H}({\rm O}_{24})$-equivariant by \cite[\S 4.4.4]{CheLan},
it thus only remains to explain which of the two eigenvectors above belongs to ${\rm M}_{\Lambda^8 V}({\rm O}_{24})$.
But Arthur's multiplicity formula for ${\rm Sp}_{16}$ (or Ikeda's results) shows that
there is no cuspidal Siegel eigenform for ${\rm Sp}_{16}(\Z)$ with standard parameter $\Delta_{17}[8] \oplus [1]$,
as explained in \cite[Example 8.5.3]{CheLan}.
This proves $\psi'_8 = \Delta_{21,13}[4] \oplus [7] \oplus [1]$ by \eqref{rallis},
hence $\psi'_{16} = \Delta_{17}[8] \oplus [7] \oplus [1]$,
and the whole table follows from \eqref{rallis} again. \end{pf}

{\small
\bibliographystyle{amsalpha}
\bibliography{poids13_leech}
}
\end{document}